\documentclass[12pt]{amsart}
\usepackage{amsmath}
\usepackage{amsthm}
\usepackage{amssymb} 
\usepackage{amscd}
\usepackage{amsfonts}
\usepackage{amsbsy} 
\usepackage{epsfig}
\usepackage{graphicx}
\usepackage{verbatim}

\newtheorem{theorem}{Theorem}
\newtheorem{remark}{Remark}
\newtheorem{proposition}{Proposition}
\newtheorem{lemma}{Lemma}
\newtheorem{corollary}{Corollary}
\newtheorem{definition}{Definition}
\newtheorem{example}{Example}

\def\Uu{{\mathfrak U}}
\def\Cc{{\mathfrak C}}
\def\Rr{{\mathfrak R}}

\def\ie{{\em i.e.,} }

\newfont\bbf{msbm10 at 12pt}
\def\eps{\varepsilon}
\def\phi{\varphi}
\def\R{{\mathbb R}}

\def\N{{\mathbb N}}
\def\Z{{\mathbb Z}}

\def\theta{\vartheta}
\def\le{\leqslant}
\def\ge{\geqslant}

\def\UIL{\underleftarrow\lim(I,T_s)}
\def\UILt{\underleftarrow\lim(I,T_{\tilde s})}
\def\CUIL{\underleftarrow\lim([c_2, c_1],T_s)}
\def\CUILt{\underleftarrow\lim([\tilde{c}_2, \tilde{c}_1],T_{\tilde s})}

\def\chain{{\mathcal C}}
\def\mesh{\mbox{mesh}}

\begin{document}

\title{The Core Ingram Conjecture for non-recurrent critical points}

\author{Ana Anu\v{s}i\'c, Henk Bruin, Jernej \v{C}in\v{c}}
\address[A.\ Anu\v{s}i\'c]{Faculty of Electrical Engineering and Computing,
University of Zagreb,
Unska 3, 10000 Zagreb, Croatia}
\email{ana.anusic@fer.hr}
\address[H.\ Bruin]{Faculty of Mathematics,
University of Vienna,
Oskar-Morgenstern-Platz 1, A-1090 Vienna, Austria}
\email{henk.bruin@univie.ac.at}
\address[J.\ \v{C}in\v{c}]{Faculty of Mathematics,
University of Vienna,
Oskar-Morgenstern-Platz 1, A-1090 Vienna, Austria} 
\email{jernej.cinc@univie.ac.at}
\thanks{AA was supported in part by Croatian Science Foundation under the project IP-2014-09-2285.
HB and J\v{C} were supported by the FWF stand-alone project P25975-N25.
We gratefully acknowledge the support of the bilateral grant \emph{Strange Attractors and Inverse Limit Spaces},  \"Osterreichische
Austauschdienst (OeAD) - Ministry of Science, Education and Sport of the Republic of Croatia (MZOS), project number HR 03/2014.}
\date{\today}


\maketitle

\begin{abstract}
We study inverse limit spaces of tent maps,
and the Ingram Conjecture, which states that the inverse limit spaces of tent
maps with different slopes are non-homeomorphic.
When the tent map is restricted to its core,
so there is no ray compactifying on the inverse limit space, this result
is referred to as the Core Ingram Conjecture.
We prove the Core Ingram Conjecture when the critical point is 
non-recurrent and not preperiodic.

\end{abstract}

\section{Introduction}\label{sec:intro}

Inverse limit spaces made their first appearance in dynamical systems in 1967 when Williams \cite{Wil1, Wil2} showed
that hyperbolic one-dimensional attractors can be represented as inverse limit spaces. The study of inverse limit spaces with the goal to describe complicated structures in strange attractors gained significance in the last two decades.
For instance, the work of Barge \& Holte \cite{BaHo} showed that for a wide range of parameters, attracting sets for maps in H\'{e}non family are homeomorphic with inverse limit spaces of unimodal interval maps.

The \emph{tent map family} $T_s:[0,1] \rightarrow [0,1]$ is defined as $T_s:=\text{ min} \{ sx,s(1-x) \}$ where $x\in [0,1]$ and $s\in (0,2]$. Let $c:= 1/2$ denote the \emph{critical point} and let $c_i:=T_s^i(c)$ for every $i\in \N$. In this paper we are concerned with the inverse limit spaces $\varprojlim([0,1], T_{s})$ using a single tent map from the parametrized family as a bonding map. It is not difficult to see that for $c\geq c_1$,
$\varprojlim([0,1], T_{s})$ is a point or an arc and thus not interesting. For the case $c\leq c_1$ it follows from Bennett's Theorem in \cite{Ben} from 1962 that we can decompose $\varprojlim([0,1], T_{s}) = \varprojlim([c_2,c_1], T_{s}) \cup \Cc$, where $\bar{0}:=(\ldots,0,0,0)\in \Cc$ is a continuous image of $[0,\infty)$ which compactifies on $\varprojlim([c_2,c_1], T_{s})$. Inverse limit space of tent map $\varprojlim([c_2,c_1], T_s)$ obtained from the forward invariant interval $[c_2,c_1]$ is called the \emph{core of the inverse limit space}. 

In the early 90's a classification problem that became known as the \emph{Ingram Conjecture} was posed:

If $1\leq s< \tilde{s}\leq 2$ then the inverse limit spaces  $\varprojlim([0,1], T_{s})$  and $\varprojlim([0,1], T_{\tilde{s}})$ are not homeomorphic.

After partial results \cite{Kail2, BJKK, Stim1, RaSt}, the Ingram Conjecture was finally answered in affirmative by Barge, Bruin \& \v Stimac in \cite{BBS}.
However, the proof  presented in \cite{BBS} crucially depends on the 
ray $\Cc$, so the core version of the Ingram Conjecture still remains open.
For H\'enon maps,  $\Cc$ plays the role of the unstable manifold of
the saddle point outside the H\'enon attractor; it compactifies on the attractor, but it is somewhat unsatisfactory to have to use this 
(and the embedding in the plane that it presupposes) for the 
topological classification.
It is also not possible to derive the core version directly from
the non-core version, because it is impossible to reconstruct $\Cc$
from the core. This is for instance illustrated by the work
of Minc \cite{Minc} showing that in general there are many non-equivalent
rays compactifying on the Knaster bucket handle continuum.

In this paper we partially solve in the affirmative 
the classification problem called the \emph{Core Ingram Conjecture}.

\begin{theorem}\label{CIC}
If $1\leq s < \tilde{s}\leq 2$ and critical points of $T_s$ and $T_{\tilde{s}}$ are non-recurrent, then the inverse limit spaces $\varprojlim([c_2,c_1], T_{s})$ and $\varprojlim([\tilde{c}_2,\tilde{c}_1], T_{\tilde{s}})$ are not homeomorphic.
\end{theorem}

If $T_s$ has a non-recurrent critical orbit, then  $\varprojlim([c_2,c_1], T_{s})$ has no endpoints. However, if the critical orbit is recurrent, $\varprojlim([c_2,c_1], T_{s})$ has endpoints (finitely many if the critical point is periodic and infinitely many if the critical orbit is infinite). For details see \cite{BaMa}. Thus, recurrent and non-recurrent case can be topologically distinguished.

Solutions to the Core Ingram Conjecture for tent maps lead to the similar conclusion for analogous question for the ''fuller'' family of unimodal maps, 
see \cite{BaDa} for details.
It turns out that Theorem~\ref{CIC} can be reduced from the case where slopes $s,\tilde{s}\in (1,2]$ to slopes $s,\tilde{s}\in (\sqrt{2},2]$, for details see \cite{BBS}.

There exist two fixed points of $T_s$: $0$ and $r:=\frac{s}{s+1}\in [c_2,c_1]$.  The \emph{arc-component} of a point $e\in \varprojlim([0,1], T_{s})$ is defined as the union of all arcs of $\varprojlim([0,1], T_{s})$ containing $e$. Let us denote the arc-component from $\varprojlim([c_2,c_1], T_{s})$ that contains $\rho:=(\ldots,r,r,r)$ by $\Rr$. It is a continuous image of the real line and is dense in both directions.
Let $\tilde\Rr \subset \varprojlim([c_2,c_1], T_{\tilde{s}})$
be the analogous arc-component to $\Rr$
that contains $\tilde \rho = (\dots, \tilde r, \tilde r, \tilde r)$
for the fixed point $\tilde{r}:=\frac{\tilde{s}}{\tilde{s}+1}$. 
The main new ingredient in this paper is:

\begin{theorem}\label{thm:R}
Let $\sqrt{2} \leq s\leq \tilde{s}\leq 2$ and assume that the
critical points of $T_s$ and $T_{\tilde{s}}$ are non-recurrent. Let $\Rr\subset \CUIL$ and $\tilde{\Rr}\subset \CUILt$ be as above. If $h:\CUIL \to \CUILt$
is a homeomorphism, then $h(\Rr)=\tilde{\Rr}$.
\end{theorem}

The main observation in our proof of this result
is Lemma~\ref{lem:12} which implies that when critical point is non-recurrent, two arcs with different $\hat{l}$-pattern also have different $l$-pattern. Lemma~\ref{lem:12} fails without the assumption that critical point is non-recurrent and this presents the main obstacle in the proof of the Core Ingram Conjecture with our approach for the case when critical point is recurrent and not periodic.

In the process of proving the Ingram Conjecture, 
partial solutions of the Core Ingram Conjecture were obtained as well.
The first result is due to Kailhofer \cite{Kail2} from 2003, who proved the Core Ingram Conjecture for the case when critical point is periodic. In 2006, Good \& Raines proved in \cite{GoRa} that the conjecture holds when critical point is non-recurrent and $\omega (c)$ is a Cantor set. However, the technique they used cannot 
be extended to all non-recurrent tent-maps as we do in this paper.
In 2007, \v Stimac \cite{Stim1} extended the mentioned result of Kailhofer and proved the Core Ingram Conjecture in the case when critical orbit is finite. Both Kailhofer and \v Stimac make use of a dense arc-component inside the core of the inverse limit space,
but not the above mentioned arc-component $\Rr$.
The most recent result regarding the Core Ingram Conjecture was obtained in 2015 by Bruin \& \v Stimac \cite{FL} who proved that conjecture holds for a set of parameters where critical point is ''extremely'' (or persistently) recurrent and not periodic. 
The last result was obtained from observations on the arc-component $\Rr$.

In this paper we prove the Core Ingram Conjecture when the critical point
is non-recurrent.
The main idea of the proof is similar as in the proof of the Ingram Conjecture in \cite{BBS}. There, the authors first assume by contradiction that there exists a 
homeomorphism between $\varprojlim([0,1], T_{s})$ and $\varprojlim([0,1], T_{\tilde{s}})$, where $s\neq \tilde{s}\in (\sqrt{2},2]$ and then it clearly follows that 
the arc-component $\Cc\subset \varprojlim([0,1], T_{s})$ maps to the arc-component $\tilde{\Cc}\subset \varprojlim([0,1], T_{\tilde{s}})$, where $\bar{0}\in \tilde{\Cc}$. 
Then they show that the concatenation of maximal link-symmetric arcs uniquely determines $\Cc$ and thus the whole $\varprojlim([0,1], T_{s})$.

The following result about the group of self-homeomorphisms 
extends as well. The proof requires only minor adjustments:
one needs to replace the arc-component $\Cc$ with $\Rr$ in the proof 
of \cite[Theorem 1.3]{BS}).

\begin{theorem}\label{thm:iso}
Assume that $T_s$ has a non-recurrent critical point.
Then for every self-homeomorphism $h:\CUIL \to \CUIL$ there is $R \in \Z$
such that $h$ and $\sigma^R$ are isotopic.
\end{theorem}

Let us give a short outline of the structure of the paper. In Section~\ref{sec:pre} we provide a basic set-up of tent maps, their inverse limit spaces and chainability. 
In Section~\ref{sec:R} we study structure of the arc-component $\Rr$.
In Section~\ref{sec:Rfixed} we prove Theorem~\ref{thm:R}. 
In Section~\ref{sec:CIC} we prove that the concatenation of maximal 
link-symmetric arcs uniquely determines $\Rr$ and thus also 
$\varprojlim([c_2,c_1], T_{s})$ for every $s\in (\sqrt{2},2]$.

\section{Preliminaries}\label{sec:pre}
\subsection{Tent maps}
Let $\N:= \{ 1,2,3,\ldots\}$ and $\N_{0}:=\{ 0,1,2,3,\ldots\}$. We define a \emph{tent map} $T_{s}:[0,1] \rightarrow [0,1]$ with slope 
$\pm s$ as $T_{s}(x)=\text{min}\{sx,s(1-x)\}$ and we restrict to $s\in(\sqrt{2},2]$. 
 Thus in particular $c_{1}=s/2$ and $c_{2}=s(1-s/2)$. 
We call the interval $[c_2,c_1]$ the \emph{core} of $T_s$.  
We restrict $T_{s}$ to the interval $I:=[0,s/2]$. 
Throughout the paper we will assume that $T_s$ has an infinite critical orbit,
because the Core Ingram Conjecture has already been proven 
for the case when $c$ is (pre)periodic, see \cite{Kail2, Stim2}.

We say that $x\in [0,1]$ is a \emph{turning point} of $T_{s}^{r}$, if there exists $m<r$ such that $T_{s}^{m}(x)=c$.
Two turning points $x,y\in [0,1]$ of $T_{s}^r$ are \emph{adjacent} if 
$T_s^{r}\mid_{[x,y]}$ is monotone.

Critical point of $T_s$ is \emph{recurrent} if for every $\eps>0$ there exists $n\in\N$ such that $|c-c_n|<\eps$. Tent map $T_s$ is called \emph{long-branched} if there exists 
$\delta>0$ such that for every $n\in\N$ if $x,y$ are two adjacent turning points of $T_s^n$, then $|T_s^n(x)-T_s^n(y)|>\delta$. Note that if $T_s$ has a non-recurrent critical point, 
then $T_s$ is long-branched.

Let $\hat{b}:=1-b$ denote the \emph{symmetric point around $c$}, where $b,\hat{b}\in [c_2,c_1]$.

\begin{lemma}\label{lem:Mis}
Let $x<y$ be adjacent turning points of $T_{s}^{r}$. Then there exists $z>y$ such that $T_{s}^{r}([y,z])=[T_{s}^{r}(x),T_{s}^{r}(y)]$.
\end{lemma}

\begin{proof}
Assume the contrary.
This leads to the following simplified statement that we need to exclude.
Let $a<b<d<e \in [0,1]$, where $b$ and $d$ are turning points of 
$T_{s}^{r}$ and  $T_{s}^{r}$ has no  other turning point in $(a,e)$. 
Assume without loss of generality that $T_{s}^{r}(a)<T_{s}^{r}(d)<T_{s}^{r}(b)<T_{s}^{r}(e)$, see Figure~\ref{fig:Not_allowed}. 
Let us assume also without loss of generality that there exists $m<n<r$ such that $T_{s}^{m}(b)=c=T_{s}^{n}(d)$. We consider the 
image of $[a,e]$ under $T_{s}^{m}$.

\textbf{Case I:} Let $|T_{s}^{m}(a)-c| \geq |T_{s}^{m}(e)-c|$. 
This means that $T_{s}^{m}(e) \in [c, \widehat{T_{s}^{m}(a)}]$,
say without loss of generality that 
$c < T_{s}^{m}(e) \leq \widehat{T_{s}^{m}(a)}$. 
Consequently, there is a point $x \in (a,b)$
 such that $T_{s}^{m}(x) = \widehat{T_{s}^{m}(d)}$,
 but then $T^n_s(x) = T^n_s(d) = c$, contradicting that
$(a,b)$ contains no turning point of $T_s^r$.

\begin{figure}[ht]
\unitlength=9mm
\begin{picture}(10,5)(1,-0.5)
\put(4,0,5){\line(0,1){4}}
\put(3.9,1.6){\line(1,0){0.2}}
\put(3.9,3.1){\line(1,0){0.2}}
\put(3.875,4.5){\line(1,0){0.25}}
\put(3.875,0.5){\line(1,0){0.25}}
\put(5,0){\line(1,0){4}}
\put(5,-0.125){\line(0,1){0.25}}
\put(9,-0.125){\line(0,1){0.25}}
\put(6.5,-0.1){\line(0,1){0.2}}
\put(7,-0.1){\line(0,1){0.2}}

\put(5,-0,5){\small $0$}
\put(6.4,-0,5){\small $x$}
\put(6.9,-0,5){\small $y$}
\put(9,-0,5){\small $1$}
\put(3.5, 0.35){\small $0$}
\put(2.75, 1.5){\small $T_{s}^r(y)$}
\put(2.75, 3.0){\small $T_{s}^r(x)$}
\put(3.5, 4.35){\small $1$}
\thicklines
\put(5.7,0.7){\line(1,3){0.8}}
\put(7,1.6){\line(-1,3){0.5}}
\put(7,1.6){\line(1,3){0.8}}
\end{picture}
\caption{Example of a pattern that is not allowed by Lemma~\ref{lem:Mis}.}
\label{fig:Not_allowed}
\end{figure}
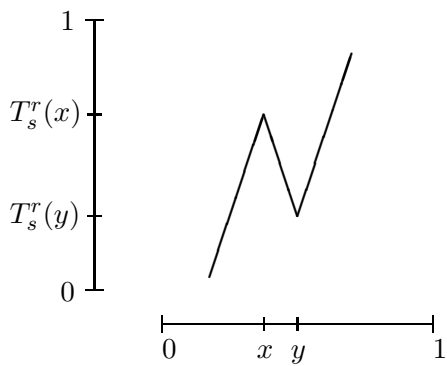

\textbf{Case II:} Let $|T_{s}^{m}(a)-c| < |T_{s}^{m}(e)-c|$. 
This means that (without loss of generality)
$c<\widehat{T_{s}^{m}(a)}<T_{s}^{m}(e)$. 
Consequently, there is a point $y \in (b,e)$ such that 
$T_{s}^{r}(y)<T_{s}^{r}(d)$ which is a contradiction with $d$ being the minimum of $T_{s}^{r}$ in $(b,e)$.
\end{proof}

\subsection{Inverse limits and chainability}
The \emph{inverse limit space} $\varprojlim(I, T_{s})$ is the collection of all backward orbits
$$
\varprojlim(I, T_{s}):=\{ e=(\ldots, e_{-2},e_{-1},e_{0}): T_{s}(e_{i-1})=e_{i}\in I \text{ for all } i\leq0 \},
$$
equipped with the product topology
and the \emph{shift homeomorphism} 
$$ 
\sigma(\ldots, e_{-2},e_{-1},e_{0})=(\ldots,e_{-2},e_{-1},e_{0},T_{s}(e_{0}))
$$ 
for every $e=(\ldots,e_{-2},e_{-1},e_{0})\in \varprojlim(I, T_{s})$.
We call $\varprojlim([c_{2},c_{1}], T_{s})$ the \emph{core of the inverse limit space}.

For $k \in \N_0$, define the \emph{k-th projection map} as 
$\pi_{k}: \varprojlim(I, T_{s}) \rightarrow I$, $\pi_{k}(e)=e_{-k}$. 
We denote an arbitrary arc-component by $\Uu$ and the arc-component that contains
a fixed point $\bar{0}=(\ldots, 0,0,0)$ by $\Cc$.
In this paper we mostly study the arc-component $\Rr$ associated with the other fixed point
$r:=\frac{s}{s+1}\in[c_{2},c_{1}]$ of $T_{s}$, so $\rho:=(\ldots,r,r,r)\in  \Rr$. 
Observe that $\rho \in\Rr\subset \varprojlim([c_{2},c_{1}], T_{s})$, while $\mathfrak{C}\nsubseteq \varprojlim([c_{2},c_{1}], T_{s})$.

\begin{definition}
The \emph{arc-length} of two points $x, y\in\Uu$ is defined as $$d(x, y):= s^k|x_{-k}-y_{-k}|,$$ where $k\in\mathbb{N}_0$ is such that $\pi_k\colon [x, y]\to[c_2, c_1]$ is injective.
\end{definition}

\begin{definition}\label{def:chain}
A space $X$ is \emph{chainable} if there are finite open covers $\chain:=\{ \ell_{i} \}_{i=1}^n$, called \emph{chains}, of arbitrarily small $\mathrm{mesh}(\chain):=\mathrm{max}_{i\in\{1,\ldots, N \}} \mathrm{diam}(\ell_{i})$ such that the \emph{links} $\ell_{i}$ 
satisfy $\ell_{i} \cap \ell_{j}\neq \emptyset$ if and only if $|i-j|\leq 1$.
Clearly the interval $[0,s/2]$ is chainable. We call $\chain_{k}$ a 
\emph{natural chain} of $\varprojlim([c_{2},c_{1}], T_{s})$ 
if for $j\in \{1,\ldots, n \}$ the following is true:
\begin{enumerate}
 \item There exists a chain $\{I_{k}^{1},I_{k}^{2}, \ldots, I_{k}^{n}\}$ of $I=[0,s/2]$ such that $\ell_{k}^{j}:=\pi_{k}^{-1}(I_{k}^{j})$ are links of $\chain_{k}$. 
 \item Each point $x\in \bigcup_{i=0}^{k}T_{s}^{-i}(c)$ is a boundary point of some link $I_{k}^{j}$.
 \item For each $i$ there is $j$ such that $T_{s}(I_{k+1}^{i})\subset I_{k}^{j}$.
\end{enumerate}
\end{definition}

We say that chain $\chain'$ \emph{refines} chain $\chain$ (written $\chain'\preceq \chain$) if for every link $\ell' \in \chain'$ there is a link $\ell\in \chain$ such that 
$\ell'\subset \ell$. Condition (3) ensures that $\chain_{k+1} \preceq \chain_k$.

\subsection{Patterns and symmetry}

\begin{definition}\label{def:k-point}
A point $x \in \CUIL$ is called a \emph{$k$-point} (with respect to the chain $\chain_k$) if there exists $n \geq 1$ such that $\pi_{k+n}(x)=c$. Note that if $c$ is not periodic, 
such $n$ is unique and we call it the \emph{$k$-level} of $x$ and denote it by $L_{k}(x)$.
\end{definition}

\begin{definition}\label{def:pattern}
Let $A$ be an arc in $\CUIL$ and assume that the number of $k$-points in $A$ is finite. Let $x_0 < x_1 < \ldots < x_N$
be the $k$-points on $A$ arranged according to the arc-length distance defined above.
Then the list of levels $L_k(x_0), L_k(x_1), \ldots, L_k(x_N)$
is called the {\em $k$-pattern} of $A$.
\end{definition}

\begin{remark}\label{rem:pattern}
From the Definition~\ref{def:pattern} it follows that $A$ is the concatenation of arcs $[x_{j-1}, x_j]$ with pairwise disjoint interiors and
$\pi_k$ maps $[x_{j-1}, x_j]$ bijectively onto $[c_{L_k(x_{j-1})}, c_{L_k(x_j)}]$.
Equivalently, if $i \in \N_0$ is such that $\pi_{k+i}:A\to \pi_{k+i}(A)$ is 
injective, then the graph $T^i|_{\pi_{k+i}(A)}$ has the same $k$-pattern as $A$.
That is, $T^i$ has turning points $\pi_{k+i}(x_0) <  \ldots 
< \pi_{k+i}(x_N)$ in
$\pi_{k+i}(A)$ and $T^i(\pi_{k+i}(x_j)) = c_{L_k(x_j)}$ for $0 \le j \le N$.
We will call this the $k$-pattern of $T^i|_{\pi_{k+i}(A)}$ as well.
\end{remark}

\begin{example}
The arc $A$ as in Figure~\ref{fig:312a} has a $k$-pattern $312$.
 
\begin{figure}[ht]
\unitlength=9mm
\begin{picture}(10,3)(1,-0.5)
\put(13,2){\line(-1,0){8}} \put(5, 2.25){\oval(0.5,0.5)[l]}
\put(1,1.5){\line(1,0){12}} \put(13, 1.75){\oval(0.5,0.5)[r]}
\put(1, 1.25){\oval(0.5,0.5)[l]}

\put(1,0){\line(1,0){12}}
\put(1,-0,25){\line(0,1){0,5}}
\put(13,-0,25){\line(0,1){0,5}}
\put(5,-0,2){\line(0,1){0,4}}

\put(-0.5,2){\vector(0,-1){2}}\put(-0.3,1){\large $\pi_{k}$}
\put(1,-0,5){\small $c_2$}
\put(5,-0,5){\small $c_3$}
\put(13,-0,5){\small $c_1$}
\put(3, 2){\small $A$}
\end{picture}
\caption{Arc $A$ with $k$-pattern 312}
\label{fig:312a}
\end{figure}
 \end{example}

\begin{definition}\label{def:sym}
We say that an arc $A:=[e, e']\subset \UIL$ is \emph{$k$-symmetric} 
if $\pi_k(e)=\pi_k(e')$ and its $k$-pattern is a palindrome.
\end{definition}

\begin{remark}
Definition~\ref{def:sym} implies that the $k$-pattern of $(e,e')$, where $A$ is a $k$-symmetric arc, is of odd length and the letter 
in the middle is the largest. This can be easily seen by considering the smallest $j>k$ such that $\pi_j:A \to [c_2, c_1]$ is injective.
\end{remark}

\begin{definition}
	Take an arc-component $\Uu$ and a natural chain $\chain_k$ of $\varprojlim([c_{2},c_{1}], T_s)$. For a $k$-point $u\in\Uu$ such that $u\in \ell\in\chain_k$ we denote the arc-component of $\ell\cap\Uu$ containing $u$ by $A^u$.
\end{definition}

\begin{remark}
 Note that the definition above makes sense since every $k$-point is contained in exactly one link of the natural chain $\chain_k$.
\end{remark}

\begin{lemma}\label{lem:dense}
Let $P$ be a $k$-pattern that appears somewhere in $\UIL$, \ie
there is an arc $A \subset \UIL$ with $k$-pattern $P$.
If arc-component $\Uu$ contains no arc with $k$-pattern $P$,
then $\Uu$ is not dense in $\UIL$.
\end{lemma}

\begin{proof}
For every pattern $P$ there exist $n\in\mathbb{N}$ and $J\subset [c_2, c_1]$ such that the graph of $T_s^n|_J$ has pattern $P$.
This means that if there is a subarc $A \subset \Uu$ such that
$\pi_{k+n}$ maps $A$ injectively onto $J$, then 
$A$ has $k$-pattern $P$. 
If $\Uu$ is dense in $\UIL$, then $\pi_{k+n}(\Uu) = [c_2,c_1]$ for every $n\in\N_{0}$. 
By Lemma~\ref{lem:Mis} there exists an arc $A'\subset \Uu$ such that $\pi_{k+n}(A')$ maps injectively onto $[c_2,c_1]$. Thus there indeed exists $A\subset A' \subset \Uu$
such that $\pi_{k+n}$ maps $A$ injectively onto $J$.
This finishes the proof.
\end{proof}

\begin{definition}\label{def:link-symmetric}
An arc $A$ is called {\em k-link-symmetric} (with respect to  the chain \ $\chain_k$) if
the list of the indices of the links it passes through is a palindrome.
This list automatically has odd length, and 
there is a unique link in the middle, the {\em midlink} $\ell$,
which contains a unique arc-component $A^{m}$ of $A \cap \ell$ corresponding
to the middle letter of the palindrome. 
We call the point in $A^{m}$ with the highest $k$-level the {\em midpoint} $m$
of $A$.
\end{definition}

\begin{remark}
The definition of {\em midpoint} $m$ above is just for completeness; since we can not topologically
distinguish points in the same arc component of
$A^{m}$, any point $x \in A^{m}$ would serve equally well.
\end{remark}

\begin{remark}
Every $k$-symmetric arc is also $k$-link-symmetric but the converse does not hold. This is one of the main obstacles in the proof of the Core Ingram Conjecture.
\end{remark}

\section{The structure of the arc-component $\Rr$}\label{sec:R}

\subsection{The arcs $A_i$}

Recall that $\Rr$ is the arc-component in $\UIL$ containing 
$\rho=(\ldots, r, r, r)$ and $\tilde{\Rr}$ is the arc-component in $\UILt$ containing $\tilde{\rho}=(\ldots, \tilde{r}, \tilde{r}, \tilde{r})$.

\begin{definition}
Let $\chain_k$ be a natural chain of $\varprojlim([c_{2},c_{1}], T_{s})$. 
We define $A_i\subset \Rr$ to be the arc-component
of $\pi_{k+i}^{-1}([c_2,\hat{c}_2])$  containing $\rho$ and let $m_i:=\pi_{k+i}^{-1}(c)\in A_i$ for every $i\in\N$.
\end{definition}

\begin{lemma}
The arc $A_i\subset \Rr$ is $k$-symmetric with midpoint $m_i$ for every $i\in\N$.
\end{lemma}

\begin{proof}
For every $i\in \N$ we obtain that $\pi_{k+i}(A_i)=[c_2,\hat{c}_2]$ injectively which is 
symmetric around $c$ and so $A_i$ is $k$-symmetric.
\end{proof}

Note that only one endpoint of $A_i$ is a $k$-point;
the other endpoint is not, although the $\pi_{k+j}$-th images of endpoints are the same for every $j < i$.

\begin{remark}
Although the arcs $A_i$ are $k$-symmetric and thus $k$-link symmetric for each $i$, 
they need not  be maximal $k$-link symmetric arcs. 
This is easiest to see in the case when $c$ is periodic. If $c$ is non-recurrent, however, the arcs $\{A_i\}_{i\in\N}$ are maximal $k$-symmetric and maximal 
$k$-link symmetric, as can be derived from Corollary~\ref{cor:2}.
\end{remark}

Define 
$$
\kappa:=\min\{i\geq 3 : c_{i}\leq c \}.
$$ 
Note that $\kappa-3$ has to be an even number or $0$, otherwise the tent map $T_{s}$ is renormalizable (which we excluded by taking the slope $s > \sqrt{2}$). 

\begin{lemma} \label{l2}
Let $\rho\in\{A_i\}_{i\in\N}\subset \Rr$ be the sequence of arcs defined as above.
Then $A_{i}\subset A_{i+2}$ for every $i\in \N$.
\end{lemma}

\begin{proof}
Note that $\pi_{k+i}(A_{i})=\pi_{k+i+2}(A_{i+2})=[c_2,\hat{c}_2]$. 
We distinguish two cases:

\textbf{Case I:} Let $c_{3}<c$.
Then $c\in \pi_{k+i+1}(A_{i+2})$ and it follows that $\pi_{k+i}(A_i)\subset \pi_{k+i}(A_{i+2})$, see Figure~\ref{fig:lem4case1}. If we combine this with the fact that $\pi_{k+i}|_{A_{i}}$ is injective it follows that $A_{i}\subset A_{i+2}$.

\begin{figure}[ht]
\unitlength=9mm
\begin{picture}(10,4)(1,-0.5)
\put(13,2.5){\line(-1,0){5}} 
\put(1,2){\line(1,0){12}} \put(13, 2.25){\oval(0.5,0.5)[r]}
\put(1, 1.75){\oval(0.5,0.5)[l]}
\put(1,1.5){\line(1,0){12}} \put(13, 1.25){\oval(0.5,0.5)[r]}
\put(13,1){\line(-1,0){5}}

\put(1,0){\line(1,0){12}}
\put(1,-0,25){\line(0,1){0,5}}
\put(13,-0,25){\line(0,1){0,5}}
\put(5.5,-0,2){\line(0,1){0,4}}
\put(8,2.3){\line(0,1){0,4}}
\put(8,0.8){\line(0,1){0,4}}
\put(8,-0,2){\line(0,1){0,4}}
\put(9,-0,2){\line(0,1){0,4}}
\put(10,-0,2){\line(0,1){0,4}}

\put(-0.5,2){\vector(0,-1){2}}\put(-0.3,1){\large $\pi_{k+i}$}
\put(1,-0,5){\small $c_2$}
\put(5.5,-0,5){\small $c$}
\put(8,-0,5){\small $c_4$}
\put(9,-0,5){\small $r$}
\put(10,-0,5){\small $\hat{c}_2$}
\put(13,-0,5){\small $c_1$}
\put(3, 2,5){\small $A_{i+2}$}

\end{picture}
\caption{The arc $A_{i+2}$ as in Case I.}
\label{fig:lem4case1}
\end{figure}
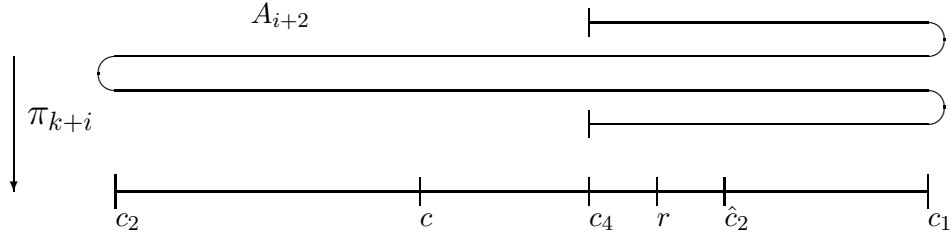

\textbf{Case II:} Let $c_{3}\geq c$.
For $s> \sqrt{2}$ we have $c_3 = T_s(\hat c_2) \leq r \leq \hat{c}_2 \leq c_4$, because $\kappa-3\in \N$ is an even number, see Figure~\ref{fig:lem4case2}.
We obtain that $\pi_{k+i}(A_{i+2})$ maps in $2$-to-$1$ fashion on the interval $[c_2,c_4]$ and because $c_4\geq\hat{c}_2$ we obtain that $\pi_{k+i}(A_{i})\subset \pi_{k+i}(A_{i+2})$ and
thus again $A_{i}\subset A_{i+2}$.
\end{proof}

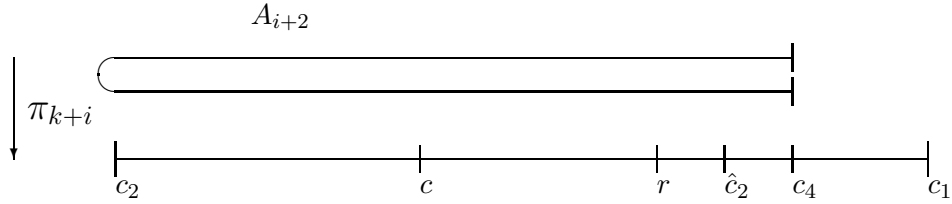
\begin{figure}[ht]
\unitlength=9mm
\begin{picture}(10,3)(1,-0.5)

\put(1,1.5){\line(1,0){10}} 
\put(1, 1.25){\oval(0.5,0.5)[l]}
\put(1,1){\line(1,0){10}} 

\put(1,0){\line(1,0){12}}
\put(1,-0,25){\line(0,1){0,5}}
\put(13,-0,25){\line(0,1){0,5}}
\put(5.5,-0,2){\line(0,1){0,4}}
\put(9,-0,2){\line(0,1){0,4}}
\put(10,-0,2){\line(0,1){0,4}}
\put(11,-0,2){\line(0,1){0,4}}
\put(11,1.3){\line(0,1){0,4}}
\put(11,0.8){\line(0,1){0,4}}

\put(-0.5,1.5){\vector(0,-1){1.5}}\put(-0.3,0.6){\large $\pi_{k+i}$}
\put(1,-0,5){\small $c_2$}
\put(5.5,-0,5){\small $c$}
\put(9,-0,5){\small $r$}
\put(10,-0,5){\small $\hat{c}_2$}
\put(11,-0,5){\small $c_4$}
\put(13,-0,5){\small $c_1$}
\put(3, 2){\small $A_{i+2}$}

\end{picture}
\caption{The arc $A_{i+2}$ as in Case II.}
\label{fig:lem4case2}
\end{figure}

In the following lemma let $A_{i,j}\subset A_i\subset \Rr$ denote the longest arc (in arc-length) such that $\rho\in A_{i,j}$ and $\pi_{k+j}: A_{i,j}\rightarrow [c_2,c_1]$ is injective for some $j\leq i$, see Figure~\ref{fig:lem5case}. Note that $A_{i,i}=A_i$.
 
\begin{lemma}\label{lem:kappa}
Let $\kappa$ and arcs $\{A_i\}_{i\in\N}$ be as defined above. Then $A_{i}\subset A_{i+\kappa}$ and $A_{i}\nsubseteq A_{i+l}$ for every $i\in \N$ and every odd $l< \kappa$. 
\end{lemma}

\begin{proof} We distinguish cases:

\begin{figure}[ht]
\unitlength=9mm
\begin{picture}(10,9)(1.5,-1)

\put(1,7.3){\line(1,0){3.5}}
\put(13,6.9){\line(-1,0){12}}\put(1,7.1){\oval(0.4,0.4)[l]}
\put(2.5,6.5){\line(1,0){10.5}}\put(13,6.7){\oval(0.4,0.4)[r]}
\put(2.5,6.3){\oval(0.4,0.4)[l]}\put(13,6.1){\line(-1,0){10.5}}
\put(1,5,7){\line(1,0){12}}\put(13,5.9){\oval(0.4,0.4)[r]}
\put(1,5,5){\oval(0.4,0.4)[l]}\put(1,5,3){\line(1,0){3,5}}
\put(4.5,5.3){\vector(1,0){0}}
\put(4.5,7.3){\vector(1,0){0}}

\put(13,4.5){\line(-1,0){2}} 
\put(1,4.1){\line(1,0){12}} \put(13, 4.3){\oval(0.4,0.4)[r]}
\put(1, 3.9){\oval(0.4,0.4)[l]}
\put(1,3.7){\line(1,0){12}} \put(13, 3.5){\oval(0.4,0.4)[r]}
\put(13,3.3){\line(-1,0){2}}
\put(11,4.5){\vector(-1,0){0}}
\put(11,3.3){\vector(-1,0){0}}

\put(13,1.7){\line(-1,0){10.5}}
\put(1,1.3){\line(1,0){12}}\put(13,1.5){\oval(0.4,0.4)[r]}
\put(2.5,1.3){\vector(-1,0){0}}
\put(1,1.3){\vector(-1,0){0}}
\put(2.5,1.7){\vector(-1,0){0}}
\put(10,1.3){\vector(1,0){0}}

\put(1,0.3){\line(1,0){9}}
\put(1,0.3){\vector(-1,0){0}}
\put(10,0.3){\vector(1,0){0}}

\put(1,0){\line(1,0){12}}
\put(1,-0,2){\line(0,1){0,4}}
\put(13,-0,2){\line(0,1){0,4}}
\put(5.5,-0,1){\line(0,1){0,2}}
\put(11,-0,1){\line(0,1){0,2}}
\put(9,-0,1){\line(0,1){0,2}}
\put(10,-0,1){\line(0,1){0,2}}
\put(2.5,-0,1){\line(0,1){0,2}}
\put(4.5,-0,1){\line(0,1){0,2}}

\put(1,1){\line(1,0){12}}
\put(1,0,8){\line(0,1){0,4}}
\put(13,0,8){\line(0,1){0,4}}
\put(5.5,0,9){\line(0,1){0,2}}
\put(11,0,9){\line(0,1){0,2}}
\put(9,0,9){\line(0,1){0,2}}
\put(10,0,9){\line(0,1){0,2}}
\put(2.5,0,9){\line(0,1){0,2}}
\put(4.5,0,9){\line(0,1){0,2}}

\put(1,3){\line(1,0){12}}
\put(1,2,8){\line(0,1){0,4}}
\put(13,2,8){\line(0,1){0,4}}
\put(5.5,2,9){\line(0,1){0,2}}
\put(11,2,9){\line(0,1){0,2}}
\put(9,2,9){\line(0,1){0,2}}
\put(10,2,9){\line(0,1){0,2}}
\put(2.5,2,9){\line(0,1){0,2}}
\put(4.5,2,9){\line(0,1){0,2}}

\put(1,5){\line(1,0){12}}
\put(1,4,8){\line(0,1){0,4}}
\put(13,4,8){\line(0,1){0,4}}
\put(5.5,4,9){\line(0,1){0,2}}
\put(11,4,9){\line(0,1){0,2}}
\put(9,4,9){\line(0,1){0,2}}
\put(10,4,9){\line(0,1){0,2}}
\put(2.5,4,9){\line(0,1){0,2}}
\put(4.5,4,9){\line(0,1){0,2}}

\put(-1,0){\small $\pi_{k+i+3}$}
\put(1,-0,3){\scriptsize $c_2$}
\put(5.5,-0,3){\scriptsize $c$}
\put(11,-0,3){\scriptsize $c_4$}
\put(9,-0,3){\scriptsize $r$}
\put(10,-0,3){\scriptsize $\hat{c}_2$}
\put(13,-0,3){\scriptsize $c_1$}
\put(2.5,-0,3){\scriptsize $c_3$}
\put(4.5,-0,3){\scriptsize $c_5$}

\put(-1,1){\small $\pi_{k+i+2}$}
\put(1,0,7){\scriptsize $c_2$}
\put(5.5,0,7){\scriptsize $c$}
\put(11,0,7){\scriptsize $c_4$}
\put(9,0,7){\scriptsize $r$}
\put(10,0,7){\scriptsize $\hat{c}_2$}
\put(13,0,7){\scriptsize $c_1$}
\put(2.5,0,7){\scriptsize $c_3$}
\put(4.5,0,7){\scriptsize $c_5$}

\put(-1,3){\small $\pi_{k+i+1}$}
\put(1,2,7){\scriptsize $c_2$}
\put(5.5,2,7){\scriptsize $c$}
\put(11,2,7){\scriptsize $c_4$}
\put(9,2,7){\scriptsize $r$}
\put(10,2,7){\scriptsize $\hat{c}_2$}
\put(13,2,7){\scriptsize $c_1$}
\put(2.5,2,7){\scriptsize $c_3$}
\put(4.5,2,7){\scriptsize $c_5$}

\put(-1,5){\small $\pi_{k+i}$}
\put(1, 4,7){\scriptsize $c_2$}
\put(5.5,4,7){\scriptsize $c$}
\put(11,4,7){\scriptsize $c_4$}
\put(9,4,7){\scriptsize $r$}
\put(10,4,7){\scriptsize $\hat{c}_2$}
\put(13,4,7){\scriptsize $c_1$}
\put(2.5,4,7){\scriptsize $c_3$}
\put(4.5,4,7){\scriptsize $c_5$}

\put(14, 0.1){\vector(0, 1){0.8}}
\put(14, 1.6){\vector(0, 1){0.8}}
\put(14, 3.6){\vector(0, 1){0.8}}
\put(13.5,0.35){\small $T_s$}
\put(13.5, 1.8){\small $T_s$}
\put(13.5, 3.8){\small $T_s$}

\put(7,0.4){\scriptsize $A_{i+3}$}
\put(11.5,1.8){\scriptsize $A_{i+3}$}
\put(5.5,1.4){\scriptsize $A_{i+2}$}
\put(11.5,3.8){\scriptsize $A_{i+3}$}
\put(1.5,6.2){\scriptsize $A_{i+3}$}
\put(7,7){\scriptsize $A_{i}$}
\put(9,0.45){\scriptsize $\rho$}
\put(9,1.45){\scriptsize $\rho$}
\put(9,3.85){\scriptsize $\rho$}
\put(9,7.05){\scriptsize $\rho$}
\put(9,0.3){\circle*{0.08}}
\put(9,1.3){\circle*{0.08}}
\put(9,3.7){\circle*{0.08}}
\put(9,6.9){\circle*{0.08}}

\put(6 ,-0.3){\scriptsize $\pi_{k+i+3}(A_{i+3,i+3})$}
\put(6.3 ,4.7){\scriptsize $\pi_{k+i}(A_{i+3,i})$}

\put(1, 5){\vector(-1, 0){0}}
\put(1, 0){\vector(-1, 0){0}}
\put(13, 5){\vector(1, 0){0}}
\put(10, 0){\vector(1, 0){0}}
\put(1, 6.9){\vector(-1, 0){0}}
\put(10, 6.9){\vector(1, 0){0}}

\thicklines
\put(1,0){\line(1,0){9}}
\put(1,5){\line(1,0){12}}
\put(1,6.9){\line(1,0){9}}
\put(1,1.3){\line(1,0){9}}

\end{picture}
\caption{Arcs $A_i, A_{i+1}, A_{i+2}, A_{i+3}$ in mentioned projections as in Case I.}
\label{fig:lem5case}
\end{figure}
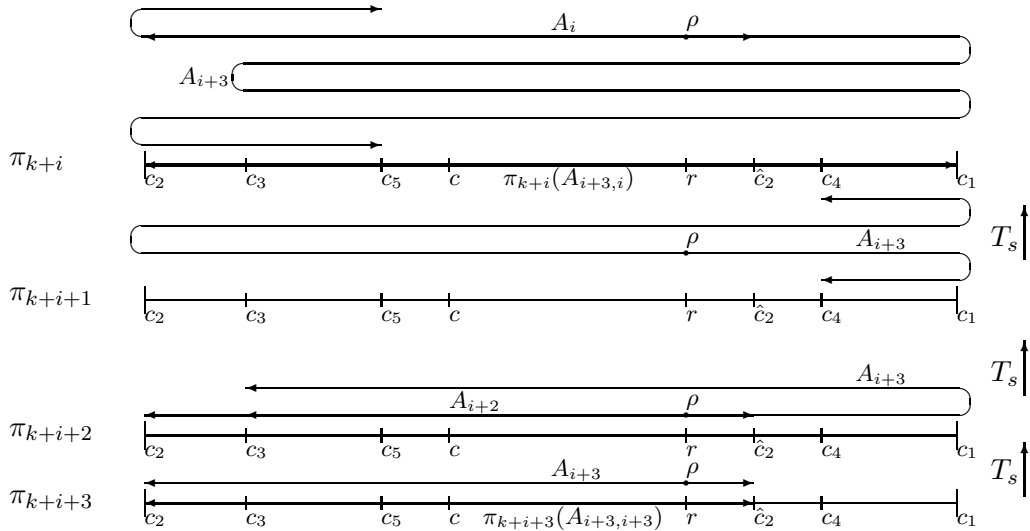

\textbf{Case I:} First assume that $\kappa=3$.
Since $T_{s}(c_2)=c_3>c_2$ it follows that $\pi_{k+i}(A_{i})\nsubseteq \pi_{k+i}(A_{i+1})$ and thus $A_{i}\nsubseteq A_{i+1}$.
Furthermore, $[c,c_1]\subset \pi_{k+i+2}(A_{i+3,i+2})$. This means that $\pi_{k+i+1}(A_{i+3,i+1})=[c_2,c_1]$ and hence
$\pi_{k+i}(A_{i+3,i})=[c_2,c_1]$, see Figure~\ref{fig:lem5case}. We conclude that $A_{i}\subset A_{i+3}$.
This finishes the proof for $\kappa=3$.

\textbf{Case II:} Let $\kappa\geq 5$.
Note that $c_{\kappa}<c < c_i$ for every $i\in \{3,\ldots,\kappa-1\}$. Thus we observe that $[c_2,\hat{c}_2]\nsubseteq \pi_{k+i+\kappa-l}(A_{i+\kappa,i+\kappa-l})$ 
for every odd $l \in \{1,\ldots \kappa-4 \}$. It follows that $A_i\nsubseteq A_{i+l}$ for every odd $l \in \{1,\ldots \kappa-4 \}$. \\
Because $T^{\kappa-2}_{s}(c_2)=c_{\kappa}$ it follows that $[c,c_1]\subset \pi_{k+i+2}(A_{i+\kappa,i+2})$ and $k+i+2$ is the smallest such index.
However, because $c$ is not periodic, $c_2< c_{\kappa}$ and it follows that $\pi_{k+i+2}(A_{i+2})=[c_2,\hat{c}_2]\nsubseteq \pi_{k+i+2}(A_{i+\kappa,i+2})$ and thus also
$\pi_{k+i}(A_{i})\nsubseteq \pi_{k+i}(A_{i+\kappa-2,i})$, so $A_{i}\nsubseteq A_{i+\kappa-2}$. 
As above we observe that $\pi_{k+i}(A_{i+\kappa,i})=[c_2,c_1]$ and thus $A_i\subset A_{i+\kappa}$. 
\end{proof}

\begin{lemma}\label{lem:rho}
Let the arcs $\{A_i\}_{i\in\N}$ with midpoints $\{m_i\}_{i\in \N}$ be as defined above.
Then $\rho \in [m_i, m_{i+1}]$ and $\rho \notin [m_i, m_{i+2}]$
for all $i \in \N$.  
\end{lemma}

\begin{proof}
Since by definition $\pi_{k+i}(A_{i})=[c_2,\hat{c}_2]$ and 
$\pi_{k+i}(m_i)=c$, we have
$\pi_{k+i}(m_{i+1})=T_{s}(\pi_{k+i}(m_i))=c_1$. Since $r\in [\pi_{k+i}(m_i),\pi_{k+i}(m_{i+1})]$ it follows that $\rho\in [m_{i},m_{i+1}]$.

To prove the second statement observe that $\pi_{k+i}(m_i)=c$ and $\pi_{k+i}(m_{i+2})=c_2$. Because 
$r\notin [c_2,c]$ and $r\in \pi_{k+i}(A_{i})\subset \pi_{k+i}(A_{i+2})$ 
 by Lemma~\ref{l2}, we obtain $\rho\notin [m_i,m_{i+2}]$.
\end{proof}

\begin{lemma}\label{lem:mid}
The arcs $\{A_i\}_{i\in\N}$ with midpoints $\{m_i\}_{i\in\N}$ as above satisfy 
 $m_{i+2} \in \partial A_i$ and $m_{i+1} \notin A_i$. 
Furthermore, if $\kappa=3$ then $m_i\in A_{i+1}$ and if $\kappa>3$ then $m_i\notin A_{i+1}$.
\end{lemma}

\begin{proof}
Since $\pi_{k+i}(A_i)=[c_2,\hat{c}_2]$ we have $\pi_{k+i}(m_{i})=c$. Moreover $T_{s}^2(c)=c_2$ and thus $\pi_{k+i}(m_{i+2})=c_2$. Because $A_i\subset A_{i+2}$ it follows that $m_{i+2}\in \partial A_{i}$.

To prove the second statement, observe that $\pi_{k+i}(m_{i+1})=c_1\notin [c_2,\hat{c}_2]$.

For the third statement first assume that $\kappa=3$; it follows that $c\in [c_3,c_1]$ and so $m_i\in A_{i+1}$.
If $\kappa>3$ then $c_3<c$ and thus $c\notin \pi_{k+i}(A_{i+1})$, so it follows that $m_i\notin A_{i+1}$.
\end{proof}

By Lemma~\ref{lem:mid} $m_{i+2} \in \partial A_i$.
We denote the other boundary point of $A_i$ by $\hat m_{i+2}$.

\begin{figure}[ht]
\unitlength=8mm
\begin{picture}(15,2)(1,0)
\thicklines
\put(-1,1){\line(1,0){19}}
\thinlines
\put(9,0.9){\line(0,1){0.2}}\put(8.9, 0.6){\small $\rho$}
\put(8,0.9){\line(0,1){0.2}}\put(7.8, 0.6){\small $m_1$}
\put(6,0.9){\line(0,1){0.2}}\put(5.8, 0.6){\small $m_3$}
\put(10,0.9){\line(0,1){0.2}}\put(9.8, 0.6){\small $\hat m_3$}
\put(3,0.9){\line(0,1){0.2}}\put(2.8, 0.6){\small $m_5$}
\put(13,0.9){\line(0,1){0.2}}\put(12.8, 0.6){\small $\hat m_5$}
\put(0,0.9){\line(0,1){0.2}}\put(-0.2, 0.6){\small $m_7$}
\put(1.5,0.9){\line(0,1){0.2}}\put(1.3, 1.2){\small $\hat{m_8}$}
\put(16,0.9){\line(0,1){0.2}}\put(15.8, 0.6){\small $\hat m_7$}
\put(11,0.9){\line(0,1){0.2}}\put(10.8, 1.2){\small $m_2$}
\put(7,0.9){\line(0,1){0.2}}\put(6.8, 1.2){\small $\hat m_4$}
\put(14,0.9){\line(0,1){0.2}}\put(13.8, 1.2){\small $m_4$}
\put(5,0.9){\line(0,1){0.2}}\put(4.8, 1.2){\small $\hat m_6$}
\put(17,0.9){\line(0,1){0.2}}\put(16.8, 1.2){\small $m_6$}
\end{picture}
\caption{The structure of the arc-component $\Rr$; $A_i=[m_{i+2}, \hat{m}_{i+2}]$ has midpoint $m_i$.} \label{fig:orderA_i}
\end{figure}
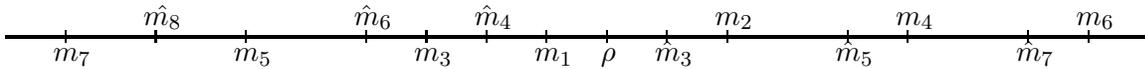

Note that all properties of the $k$-link symmetric arcs $\{A_i\}_{i\in\N}$ proved in this section are topological, meaning they are preserved under a
homeomorphism.

\subsection{$\eps$-symmetry}

\begin{definition}\label{def:eps-sym}
Let $J:=[a, b]\subset [c_2,c_1]$ be an interval. The map $f\colon J\to \mathbb{R}$ is called {\em$\eps$-symmetric} if there is a continuous bijection $x\mapsto i(x)=:\hat{x}$ swapping $a$ and $b$, such that $|f(x) - f(\hat x)|<\eps$ for all $x\in J$.
Note that $i\colon J\to J$ has a unique fixed point $m$. We say that $f$ is $\eps$-symmetric with {\em center} $m$ (or just $\eps$-symmetric {\em around} $m$).
\end{definition}

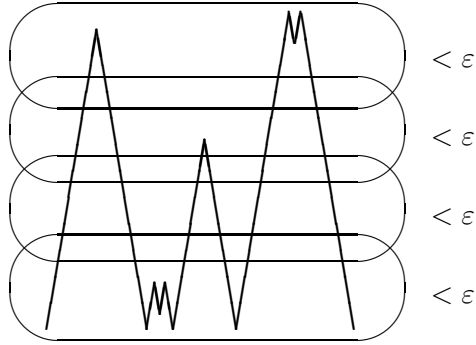
\begin{figure}[ht]
\unitlength=7mm
\begin{picture}(15,6)(0,0.3)
\thicklines
\put(3.7,0){\line(1,6){0.95}}
\put(5.6,0){\line(-1,6){0.95}}
\put(5.6,0){\line(1,6){0.15}}
\put(5.85,0.3){\line(-1,6){0.1}}
\put(5.85,0.3){\line(1,6){0.1}}
\put(6.1,0){\line(-1,6){0.15}}
\put(6.1,0){\line(1,6){0.6}}
\put(7.3,0){\line(-1,6){0.6}}
\put(7.3,0){\line(1,6){1}}
\put(8.4,5.45){\line(-1,6){0.1}}
\put(8.423,5.45){\line(1,6){0.1}}
\put(9.535,0){\line(-1,6){1}}

\thinlines
\put(8,5.2){\oval(10,2)[l]}
\put(8,5.2){\oval(5,2)[r]}
\put(8,3.8){\oval(10,2)[l]}
\put(8,3.8){\oval(5,2)[r]}
\put(8,2.3){\oval(10,2)[l]}
\put(8,2.3){\oval(5,2)[r]}
\put(8,0.8){\oval(10,2)[l]}
\put(8,0.8){\oval(5,2)[r]}
\put(11,5){\small $<\eps$}
\put(11,3.5){\small $<\eps$}
\put(11,2){\small $<\eps$}
\put(11,0.5){\small $<\eps$}
\end{picture}
\caption{A graph of $\eps$-symmetric map.}
\label{fig:eps-sy}
\end{figure}

\begin{remark}
Let $A\subset\Rr$ be a $k$-link symmetric arc with midpoint $m_A$, where $\mesh(\chain_k)<\eps.$
If $i$ is such that $\pi_{k+i}|_A\colon A\to \pi_{k+i}(A)$ is injective, then $T^i|_{\pi_{k+i}(A)}$ is $\eps$-symmetric around $\pi_{k+i}(m_A)$, see Figure~\ref{fig:eps-sy}.
\end{remark}

Next we restate Proposition 3.6.\ from \cite{BBS}
although the definition of $\eps$-symmetry is slightly generalized here. However, all arguments in the proof of Proposition 3.6. from \cite{BBS} with the new definition of $\eps$-symmetry 
remain the same.

\begin{proposition}\label{prop:3.6}
For every $\delta>0$ there exists $\eps>0$  such that for every interval $H=[a, b] \owns m$ such that $|m-c|, |c-a|, |c-b|>\delta$, $T^n|_H$ is not $\eps$-symmetric around $m$ 
for every $n\geq 0$. 
\end{proposition}

The next proposition and corollary rely on the non-recurrence of 
the critical point. Although Section~\ref{sec:complete}
does not need the non-recurrence, Section~\ref{sec:Rfixed} again relies on this assumption.

\begin{proposition}\label{prop:eps-symmetric}
Assume that $s\in (\sqrt{2},2]$ is such that $c$ is non-recurrent and not preperiodic.
For every $\delta > 0$ there exists $\eps > 0$ with the following property:
If $c\in J \subset [c_2,c_1]$ is an interval with midpoint $m$,
and $|c-\partial J| \ge \delta$, then
for each $n \ge 0$, either
$T_s^n|_J$ is not $\eps$-symmetric or
$|c-m| \le \eps s^{-n}$.
\end{proposition}

\begin{proof}\label{cor:eps-symmetric}
Fix $\delta>0$.
In the case that $|c-m| > \delta$, this is Proposition~\ref{prop:3.6}.
Therefore assume that  $\eps s^{-n} < |c-m| \le \delta$.
Because $c$ is not recurrent we can find $\eps>0$ so small that the map $T_s^n$ 
is monotone on a one-sided neighbourhood of $c$ of length 
$\eps s^{-n}$, and maps it therefore onto an interval
of length $\eps$. This means that $T_s^n([c,m])$ has length $\eps$,
so that $c$ and $m$ must be distinct centres of $\eps$-symmetry of $T_s^n$. 
Define the reflection around $a \in \R$ as $R_a(x) := 2a-x$,
then 
$c' := R_m(c) \in J$ is another center of $\eps$-symmetry,
and so is $c'' := R_{c'}(c)$. We continue this way until we find
a center of $\eps$-symmetry $m'$ such that $|c-m'| > \delta$
and apply Proposition~\ref{prop:3.6}.
\end{proof}

From now on assume that $\sqrt{2} < s\neq \tilde{s} \le 2$ and that the tent maps $T_s$ and $T_{\tilde{s}}$ have non-recurrent infinite critical orbits. 

Let $\tilde{c}$ denote the critical point of the map $T_{\tilde{s}}$ and let $\tilde{c}_i:= T_{\tilde{s}}^{i}$ for $i\in \N$.
Set
\begin{equation}\label{eq:delta}
\begin{array}{r}
\delta < \frac{1}{100}\min\Big\{ |c-c_i|, |\tilde c-\tilde c_i|, |c-r|, |\tilde c - \tilde r|, |T^{n}_{\tilde{s}}(x)-T^{n}_{\tilde{s}}(y)| :\text{for every }\\
 n,i\in \N 
\text{ and adjacent turning points
$x$ and $y$ of } T^n_{\tilde{s}} \Big\}. 
\end{array}
\end{equation}

\begin{corollary}\label{cor:2}
Suppose that $c$ is non-recurrent.
If $i \ge 1$ and $J \supset (c_i-\delta, c_i + \delta)$,
then $T_s^n|_J$ is not $\eps$-symmetric with midpoint $c_{i}$
for any $n \ge 0$.
\end{corollary}

\begin{proof}
Let $r \ge 0$ be minimal such that $c_{-r}\in J$. Then $T^{r}_{s}$
maps $J$ injectively onto a neighbourhood of $c$, so we can 
apply Proposition~\ref{prop:eps-symmetric} and use the fact that $|c-c_i|> \delta$ for $i\in \N$ to finish
the proof.
\end{proof}

\subsection{Completeness of the sequence $\{ A_i\}_{i \in \N}$}
\label{sec:complete}

Note that results in this section do not require the non-recurrence assumption.

\begin{definition}
	Let $\{G_i\}_{i\in \N}\subset \Uu$ be a sequence of $k$-link symmetric arcs with midpoints $m_i$ respectively and $x\in G_{i}$, where $x\in \Uu$, for all
	$i\in \N$. The sequence $\{G_i\}_{i\in \N}$ is called \emph{complete} with respect to  $x$ if every $k$-link symmetric arc $G\ni x$ not contained in a single link of a chain $\chain_k$ has midpoint in $\{m_i\}_{i\in \N}$. 
\end{definition}

\begin{proposition}
The sequence of $k$-link symmetric arcs $\{A_i\}_{i\in \N}$ is a 
complete sequence of $k$-link symmetric arcs with respect to $\rho$.
\end{proposition} 

\begin{proof}
	Assume that there exists a $k$-link symmetric arc $A\ni\rho$ not contained in a single link of $\mathcal{C}_k$, such that its midpoint $m\neq m_i$ for every $i\in \N$. Without loss of generality we can take $m$ closest to $\rho$ (in arc-length) 
	among all midpoints of such arcs. 
 Since $m$ is a $k$-point and there are no $k$-points in $(m_1, m_2)$ we obtain that $m\notin (m_1,m_2)$.
 Thus by Lemma~\ref{lem:rho} there exists $i\in\mathbb{N}$ such that $m\in(m_{i+2},m_i)$. Denote by $\ell_{p_0}, \ldots, \ell_{p_u}, \ldots, \ell_{p_v}, \ldots, \ell_{2p_u}$ the subsequent links containing arc $[m_{i+2},\hat{m}_{i+2}]$, where $m_i\in\ell_{p_u}$, $m\in\ell_{p_v}$, $m_{i+2}\in\ell_{2p_u}$ and $\hat{m}_{i+2}\in \ell_{p_0}$. Note that $\ell_{p_u+p_n}=\ell_{p_u-p_n}, $ for every $n\in\{0, \ldots, u\}$.
	
{\bf Case I:} $p_v-p_u\leq 2p_u-p_v$ (the number of links the arc $[m, m_i]$ goes through is smaller than the number of links the arc $[ m_{i+2},m]$ goes through).
	
Let $a, b$ be the boundary points of $A \cap [m_{i+2},\hat m_{i+2}]$
such that $ m_{i+2} \le b < m_i < a \le \hat m_{i+2}$, where $<$ denotes linear (arc-length) order on $\Uu$.
Observe that $d(b,m_i)>d(a,m_i)$.  
Given $x\in[m_{i+2},\hat{m}_{i+2}]$, let $\hat{x}\in[m_{i+2},\hat{m}_{i+2}]$ 
be such that $[\hat{x}, x]$ is $k$-symmetric with midpoint $m_i$.
Define an arc $\hat{A}:=[\hat{a},\hat{b}]$, see Figure~\ref{fig:case1}.
	
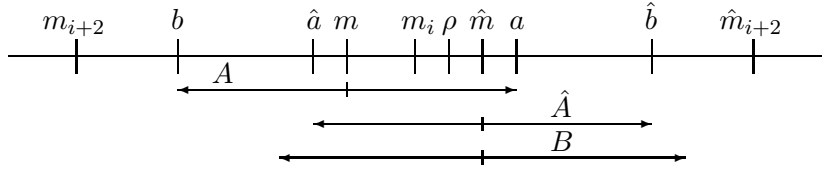
\begin{figure}[ht]
		\unitlength=9mm
		\begin{picture}(10,2.6)(1,-0.8)
		
		\put(0,1){\line(1,0){12}}
		
		\put(1,0.75){\line(0,1){0.5}}
		\put(2.5,0.75){\line(0,1){0.5}}
		\put(6,0.75){\line(0,1){0.5}}
		\put(5,0.75){\line(0,1){0.5}}
		\put(4.5,0.75){\line(0,1){0.5}}
		\put(7.5,0.75){\line(0,1){0.5}}
		\put(7,0.75){\line(0,1){0.5}}
		\put(6.5,0.75){\line(0,1){0.5}}
		\put(9.5,0.75){\line(0,1){0.5}}
		\put(11,0.75){\line(0,1){0.5}}
		
		\put(4.8,1.4){\small $m$}
		\put(4.4,1.4){\small $\hat{a}$}
		\put(0.5,1.4){\small $m_{i+2}$}
		\put(2.4,1.4){\small $b$}
		\put(5.8,1.4){\small $m_i$}
		\put(6.8,1.4){\small $\hat{m}$}
		\put(7.4,1.4){\small $a$}
		\put(6.4,1.4){\small $\rho$}
		\put(9.4,1.4){\small $\hat{b}$}
		\put(10.5,1.4){\small $\hat{m}_{i+2}$}
		
		\put(2.5,0.5){\line(1,0){5}}
		\put(4.5,0){\line(1,0){5}}
		\put(4,-0.5){\line(1,0){6}}
		
		\put(2.5,0.5){\vector(-1,0){0}}
		\put(7.5,0.5){\vector(1,0){0}}
		\put(4.5,0){\vector(-1,0){0}}
		\put(9.5,0){\vector(1,0){0}}
		\put(4,-0.5){\vector(-1,0){0}}
		\put(10,-0.5){\vector(1,0){0}}
		
		\put(3,0.6){\small $A$}
		\put(8,0.1){\small $\hat{A}$}
		\put(8,-0.4){\small $B$}
		
		\put(5,0.4){\line(0,1){0.2}}
		\put(7,-0.1){\line(0,1){0.2}} 
		\put(7,-0.6){\line(0,1){0.2}}
		\end{picture}
		\caption{Case I of the proof.}
		\label{fig:case1}
	\end{figure}
	
	Let $B$ be the maximal $k$-link symmetric arc with midpoint $\hat{m}$. Observe that $\hat{A}\subset B$ because $\hat{A}$ is a reflection of $A\cap[m_{i+2},\hat{m}_{i+2}]$ over $m_i$.
	Since $d(\rho,\hat b) > d(\rho,a)$ and $\hat a \in (m_{i+2},m_i)$, we get that $B\supset\hat{A}\supset(m_i,a)\ni\rho$.
	
By the minimality of $m$ there exists $j'<i$ such that $\hat{m}=m_{j'}$.
	
Now we study $\pi_i(A)$, see Figure~\ref{fig:proj}. 
Since $\rho \in A$, $\pi_i(m)\in(c_2, c)$ and $r\in\pi_i(A)$. 
	
	\begin{figure}[ht]
		\unitlength=9mm
		\begin{picture}(10,5)(2,0)
		
		\put(0,4){\line(1,0){14}}
		
		\put(2,3.75){\line(0,1){0.5}}
		\put(6,3.75){\line(0,1){0.5}}
		
		\put(5,3.75){\line(0,1){0.5}}
		\put(4.9,4.4){\small $m$}
		
		\put(7,3.75){\line(0,1){0.5}}
		\put(7.5,3.75){\line(0,1){0.5}}
		\put(10,3.75){\line(0,1){0.5}}
		
		\put(1.5,4.4){\small $m_{i+2}$}
		\put(5.8,4.4){\small $m_i$}
		\put(6.8,4.4){\small $\hat{m}$}
		\put(7.3,4.4){\small $\rho$}
		\put(9.8,4.4){\small $\hat{m}_{i+2}$}
		
		\put(7,3.2){\vector(0,-1){1}}
		\put(7.2,2.7){\small $\pi_i$}
		
		\put(2,1){\line(1,0){10}}
		
		\put(2,0.75){\line(0,1){0.5}}
		\put(5,0.75){\line(0,1){0.5}}
		\put(6,0.75){\line(0,1){0.5}}
		\put(7,0.75){\line(0,1){0.5}}
		\put(7.5,0.75){\line(0,1){0.5}}
		\put(10,0.75){\line(0,1){0.5}}
		\put(12,0.75){\line(0,1){0.5}}
		
		\put(1.9,0.4){\small $c_2$}
		\put(4.5,1.4){\small $\pi_i(m)$}
		\put(5.9,0.4){\small $c$}
		\put(6.5,1.4){\small $\pi_i(\hat{m})$}
		\put(7.4,0.4){\small $r$}
		\put(9.9,0.4){\small $\hat{c}_2$}
		\put(11.9,0.4){\small $c_1$}
		
		\put(5,0.7){\vector(1,0){1}}
		\put(5,0.7){\vector(-1,0){0}}
		
		\put(5.4,0.3){\small $\delta$}
		
		\put(7.7,4){\vector(1,0){0}}
		\put(7.7,1){\vector(1,0){0}}
		\put(2.1,4){\vector(-1,0){0}}
		\put(2.1,1){\vector(-1,0){0}}
		
		\put(3,3.6){\small $A$}
		\put(2.9,0.6){\small $\pi_{i}(A)$}
		
		\thicklines
		\put(2.1,1){\line(1,0){5.6}}
		\put(2.1,4){\line(1,0){5.6}}
		
		\end{picture}
		\caption{Arc $A$ in projection $\pi_{i}$ as in Case I of the proof.}
		\label{fig:proj}
	\end{figure}
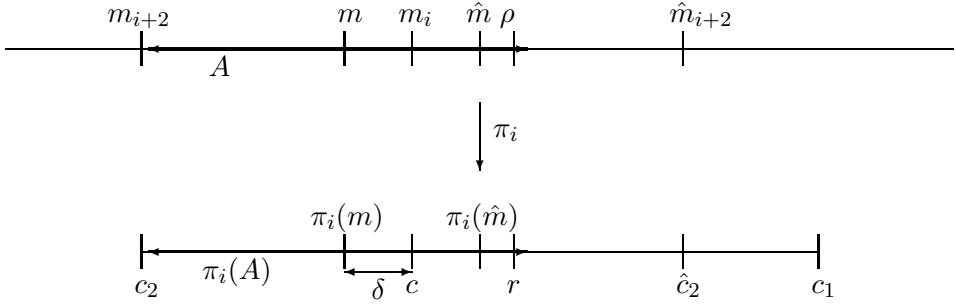
	
	If $|\pi_i(m)-c|>\delta$, we use Proposition~\ref{prop:3.6} to conclude that $A$ is not $k$-link symmetric, a contradiction.
	Assume that $|\pi_i(m)-c|\leq\delta$.
	Note that then also $|\pi_{i}(\hat{m})-c|\leq \delta$. However, $\hat{m}=m_{j'}$ for some ${j'}<i$. Note that $|\pi_i(m)-c| = |\pi_i(\hat m)-c| = |\pi_i(m_{j'})-c|\geq|\pi_i(m_{i-2})-c| > \delta$, a contradiction. The last inequality follows from the fact that $\pi_i(m_{i-2})\in T^{-2}(c)$ and the definition of $\delta$.
	
{\bf Case II:} $p_v-p_u> 2p_u-p_v$ (number of links the arc $[m,m_i]$ goes through is larger than the number of links the arc $[m_{i+2},m]$ goes through).

Here let $\hat{x}$ denote the point such that $[x,\hat{x}]$ is $k$-symmetric with midpoint $m_{i+2}$ for some $x\in [m_{i+4},\hat{m}_{i+4}]$ and $r_m(y)$ the 
	$k$-point with the largest $k$-level (in its link) such that $[y,r_m(y)]$ is $k$-link symmetric with midpoint $m$ for some $y\in [m_{i+2},\hat{m}_{i+2}]$.
	
	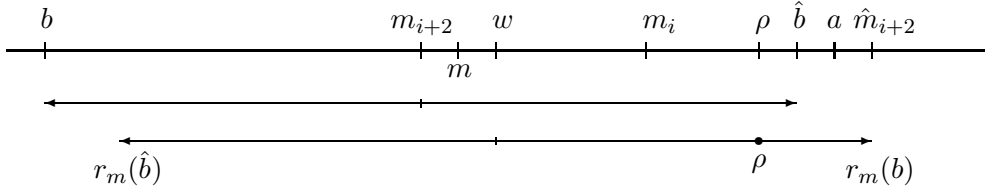
\begin{figure}[ht]
		\unitlength=5mm
		\begin{picture}(10,5)(8,-0.3)
		
		\put(0,3.2){\line(1,0){26}}
		
		\put(1,3){\line(0,1){0.4}}
		\put(11,3){\line(0,1){0.4}}
		\put(12,3){\line(0,1){0.4}}
		\put(13,3){\line(0,1){0.4}}
		\put(17,3){\line(0,1){0.4}}
		\put(20,3){\line(0,1){0.4}}
		\put(21,3){\line(0,1){0.4}}
		\put(22,3){\line(0,1){0.4}}
		\put(23,3){\line(0,1){0.4}}
		
		\put(0.9,3.8){\small $b$}
		\put(10.2,3.8){\small $m_{i+2}$}
		\put(11.7,2.5){\small $m$}
		\put(12.9,3.8){\small $w$}
		\put(16.9,3.8){\small $m_i$}
		\put(19.9,3.8){\small $\rho$}
		\put(20.9,3.8){\small $\hat{b}$}
		\put(21.8,3.8){\small $a$}
		\put(22.5,3.8){\small $\hat{m}_{i+2}$}
		
		\put(1,1.8){\vector(1,0){20}}
		\put(1,1.8){\vector(-1,0){0}}
		\put(11,1.7){\line(0,1){0.2}}
		
		\put(3,0.8){\vector(1,0){20}}
		\put(3,0.8){\vector(-1,0){0}}
		\put(13,0.7){\line(0,1){0.2}}
		
		\put(20,0.8){\circle*{0.2}}
		\put(19.8,0.1){\small $\rho$}
		\put(2.3,-0.2){\small $r_m(\hat{b})$}
		\put(22.3,-0.2){\small $r_m(b)$}
		
		\end{picture}
		\caption{Reflections as in Case II of the proof.}
		\label{fig:case2}
	\end{figure}

	Let $b$ be the endpoint of $A \cap [m_{i+4},m_{i+2}]$ that is the furthest
	away from $\rho$.
	Take $w:=r_m({m}_{i+2})$ and note that $w\in(m,m_i)$ by the assumption 
	for Case II. We reflect the arc $[b,\hat{b}]$ over $m$ and obtain an arc $[ r_m(\hat{b}),r_m(b)]$ (see Figure~\ref{fig:case2}) which is $k$-link symmetric with midpoint $w$. 
Since $\hat\rho\in[m_{i+2}, m_{i+4}]$, we have that  $r_m(\rho)\in[m_{i+2}, m_{i+4}]$ and thus $r_m(\rho)\in[m_{i+2}, b]$. We conclude that 
	$\rho\in [r_m(\hat{b}),r_m(b)]$ which is a contradiction with the minimality of $m$, because we found a $k$-link symmetric arc around $w$ such that 
	$[r_m(\hat{b}),r_m(b)]\ni \rho$ and $d(\rho,m)>d(\rho,w)$.
\end{proof}

\section{Arc-component $\Rr$ is fixed under homeomorphisms}

\label{sec:Rfixed}
Assume by contradiction that there exists a homeomorphism $h:\CUIL \to \CUILt$.
Our goal in this section is to prove Theorem~\ref{thm:R} (which holds
also if $s = \tilde s$).

\begin{definition}\label{def:eps-close}
We say that the maps $f:J \to \R$ and $g:K \to \R$ for intervals 
$J,K\subset [c_2,c_1]$, are {\em $\eps$-close}
if there exists a homeomorphism $h:J \to K$ such that $|f(x) - g \circ h(x)| < \eps$ for all $x \in J$, see Figure~\ref{fig:eps-close}.
\end{definition}

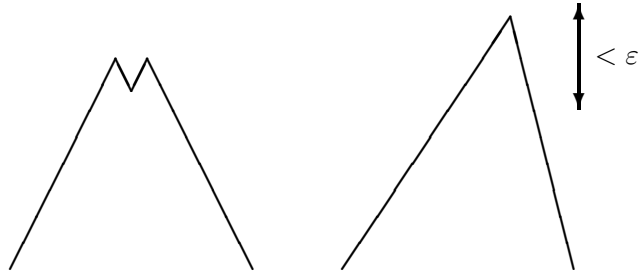
\begin{figure}[ht]
\unitlength=7mm
\begin{picture}(15,5)(-1,0.85)
\thicklines
\put(7,1){\line(2,3){3.2}}
\put(11.4,1){\line(-1,4){1.2}}
\put(0.7,1){\line(1,2){2}}
\put(5.3,1){\line(-1,2){2}}
\put(3,4.4){\line(-1,2){0.3}}
\put(3,4.4){\line(1,2){0.3}}
\put(11.5,4.15){\vector(0,1){1.9}}
\put(11.5,4.15){\vector(0,-1){0.1}}
\put(11.8,4.9){\small $<\eps$}
\end{picture}
\caption{Graphs of $\eps$-close maps.}\label{fig:eps-close}
\end{figure}

\begin{remark} Maps that are $\eps$-close can have different number of branches. However, in the non-recurrent case the number of branches must be the same. Note also that $\eps$-closeness is not an equivalence relation
because it is not transitive.
\end{remark}

From now on take $\eps = \eps(\delta) > 0$  (except in Lemma~\ref{lem:12} where $\eps$ is chosen independently) such that
Propositions~\ref{prop:3.6} and~\ref{prop:eps-symmetric} apply both for $\UIL$ and $\UILt$. 

Choose integers $\hat k, l, k$ so large that 
$\mesh(\chain_{\hat k}), \mesh(\tilde\chain_l), \mesh(\chain_k) < \eps$ 
and 
\begin{equation}\label{eq:refine}
h^{-1}(\tilde \chain_l) \preceq \chain_{\hat k} \quad \text{ and }
\quad h(\chain_k) \preceq \tilde \chain_l.
\end{equation}
Let $B_i := h(A_i)$; since $h(\chain_k)$ refines $\tilde \chain_l$,
$B_i$ is link-symmetric in $\tilde \chain_l$. 
We denote
the midpoint of $B_i$ by $n_i$, so $B_i = [n_{i+2},\hat n_{i+2}]$, see Figure~\ref{fig:order}. Let $q := h(\rho)$.

\begin{lemma}
	The sequence $\{B_{i}\}_{i\in \N}\subset \varprojlim([c_{2},c_{1}], T_{\tilde{s}})$ is a complete sequence of $l$-link symmetric arcs with respect to $q$.
\end{lemma}

\begin{proof}	
	Assume by contradiction that there exists an $l$-link-symmetric arc $B \owns q$
	with midpoint $n\in h(\Rr)$ such that $\pi_{l}(B)$ is not injective, $n\neq n_i$ for every $i\in \N$ and $B$ is not contained in a single link of the chain 
	$\tilde{\mathcal{C}}_l$. Take $B$ such that $n$ is the closest to $q$ (in arc-length) with the above properties.
Assume there exists $j\in \N$ such that $n \in (n_j,n_{j+2})$.

Recall that for a $\hat k$-point $u\in\ell\in{\chain}_{\hat k}$ we denote the arc-component of $u$ in $\ell$ by $A^u$.
		
Because we chose chains such that 
$h^{-1}(\tilde\chain_l) \preceq \chain_{\hat k}$,
the arc $A:=h^{-1}(B)$ is $\hat k$-link symmetric and $\rho\in A$. Assume that the midpoint $m$ of $A$ is not contained 
	in $A^{m_j}$ or $A^{m_{j+2}}$, thus $m\in(m_j, m_{j+2})$.
	Note that $\chain_k=h^{-1}\circ h(\chain_k)\preceq h^{-1}(\tilde{\chain}_l)\preceq \chain_{\hat{k}}$ and thus $k\geq\hat k$. 
	Since $\{A_i\}_{i\in \N}$ is a complete sequence of $k$-link symmetric arcs with respect to $\rho$, we get that $\{A_{i+ k-\hat k}\}_{i\in \N}$ is a complete sequence of 
	$\hat k$-link symmetric arcs with respect to $\rho$. Since $A$ is $\hat k$-link symmetric and $\rho\in A$, by the choice of $\eps>\textrm{mesh}\, 
	\chain_{\hat{k}}$ we obtain that $A$ is not contained in a single link of $\chain_{\hat k}$.
Thus $m=m_{i+ k-\hat k}$ for some $i\geq 1$. 
But $m=m_{i+ k-\hat k}\in(m_j, m_{j+2})$ gives a contradiction. 
	
	Assume that $m\in A^{m_j}$. Recall that $r_x$ denotes the reflection over $x$, that is, $r_x(y)$ is a point such that 
$[r_x(y), y]$ is $l$-link-symmetric with midpoint $x$
(and if we can choose $r_x(y)$ to be an $l$-point, we choose the one
with the highest $l$-level).
\newline
Since $B_j=[\hat n_{j+2}, n_{j+2}]$ is $l$-link symmetric with midpoint $n_j$ and $n\in(n_j, n_{j+2})$,  $r_{n_j}(n)$ and $n$ are contained in the same link of $\tilde\chain_l$. 
Since $h^{-1}(n)\in A^m=A^{m_j}$, and $h^{-1}(\tilde\chain_l)\preceq \chain_{\hat{k}}$ it follows that $h^{-1}(r_{n_j}(n))\in A^{m_j}$. 
We conclude that $h^{-1}([r_{n_j}(n), n])\subset A^{m_j}$.
Since $h^{-1}(q)=\rho\notin A^{m_j}$ we obtain that $r_{n_j}(n)\in(q, n_j)$.
But then $r_{n_j}(n)$ is the midpoint of 
	the $l$-link symmetric arc which contains $q$ and thus $r_{n_j}(n)=n_i$ for some $i<j$, because we assumed that $B$ is the closest $l$-link symmetric arc such that  
	$n\notin (n_i)_{i\in \N}$. 
  But $m_i=h^{-1}(r_{n_j}(n))\in A^{m_j}$ which is impossible. 

If $m\in A^{m_{j+2}}$ then
	$r_n(n_{j+2})\in(n_j, n_{j+2})$ is a midpoint of an $l$-link symmetric arc which contains $q$, a contradiction.
	
	If $n\in (q, n_1)$ or $n\in (q, n_2)$ the proof follows similarly. 
\end{proof}

Assume by contradiction that $h(\Rr) \neq \tilde \Rr$. 
Then there is $\hat l \geq l+3$
such that $q_{\hat l+1} < \tilde{c} < q_{\hat l}$, and that 
$\pi_{\hat l}:B_1 \to [\tilde c_2,\tilde c_1]$ is injective. Condition $\hat{l}\geq l+3$ is required in the proof of Lemma~\ref{lem:Qint} (Case III).

The crux of the proof is to show that $h(\Rr)$ cannot contain the
$\hat l$-pattern $12$, and therefore (by Lemma~\ref{lem:dense})
cannot be dense in $\CUILt$, which contradicts the fact that $\Rr$ is dense in $\CUIL$.

Let $B \owns q$ be the maximal arc such that $\pi_{\hat l}:B \to  [\tilde c_2,\tilde c_1]$ is injective. Then $\pi_{\hat l}(B) \subset [\tilde c_3,\tilde c_1]$.
Indeed, since $q_{\hat l+1} < \tilde{c}$, $\pi_{\hat l+1}(B) \subset [\tilde c_2, \tilde c]$.
Hence $\pi_{\hat l}(B) \subset T_{\tilde s}([\tilde c_2,\tilde c]) = [\tilde c_3,\tilde c_1]$, see Figure~\ref{fig:pi(B)}.

\begin{figure}[ht]
\unitlength=9mm
\begin{picture}(10,5)(1,-0.5)

\put(1,3.9){\line(1,0){1}}
\put(1,3.7){\oval(0.4,0.4)[l]}
\put(1,3.5){\line(1,0){10}}
\put(11,3.3){\oval(0.4,0.4)[r]}
\put(11,3.1){\line(-1,0){1}}
\put(3.5,3.5){\circle*{0.08}}
\put(3.4,3.7){\scriptsize $q$}

\put(1,2.75){\line(1,0){10}}
\put(-0.7,3){\large $\pi_{\hat l+1}$}
\put(0.5, 3.5){\vector(0, -1){0.8}}

\put(1,2.65){\line(0,1){0.4}}
\put(0.9,2.25){\scriptsize $\tilde{c}_2$}
\put(4,2.65){\line(0,1){0.2}}
\put(3.9,2.25){\scriptsize $\tilde{c}$}
\put(11,2.55){\line(0,1){0.4}}
\put(10.9,2.25){\scriptsize $\tilde{c}_1$}
\put(3,2.65){\line(0,1){0.2}}
\put(2.9,2.25){\scriptsize $\tilde{c}_3$}
\put(3.5,2.65){\line(0,1){0.2}}
\put(3.4,3){\scriptsize $q_{\hat{l}+1}$}

\put(12,2.5){\vector(0,-1){2}}
\put(12.2,1.5){\small $T_{\tilde{s}}$}

\put(3,1.7){\line(1,0){1}}
\put(3,1.5){\oval(0.4,0.4)[l]}
\put(3,1.3){\line(1,0){8}}
\put(11,1.1){\oval(0.4,0.4)[r]}
\put(11,0.9){\line(-1,0){1}}
\put(7.5,1.3){\circle*{0.08}}
\put(7.4,1.5){\scriptsize $q$}

\put(1,0.5){\line(1,0){10}}
\put(1,0.3){\line(0,1){0.4}}
\put(0.9,0){\scriptsize $\tilde{c}_2$}
\put(4,0.4){\line(0,1){0.2}}
\put(3.9,0){\scriptsize $\tilde{c}$}
\put(11,0.3){\line(0,1){0.4}}
\put(10.9,0){\scriptsize $\tilde{c}_1$}
\put(3,0.4){\line(0,1){0.2}}
\put(2.9,0){\scriptsize $\tilde{c}_3$}
\put(7.5,0.4){\line(0,1){0.2}}
\put(7.4,0){\scriptsize $q_{\hat{l}}$}

\put(-0.3,0.8){\large $\pi_{\hat{l}}$}
\put(0.5, 1.3){\vector(0, -1){0.8}}

\put(3,0.5){\vector(-1,0){0}}
\put(11,0.5){\vector(1,0){0}}
\put(1,2.75){\vector(-1,0){0}}
\put(4,2.75){\vector(1,0){0}}

\put(3,1.3){\vector(-1,0){0}}
\put(11,1.3){\vector(1,0){0}}
\put(1,3.5){\vector(-1,0){0}}
\put(4,3.5){\vector(1,0){0}}

\put(5,0.7){\scriptsize $\pi_{\hat{l}}(B)$}
\put(2,2.95){\scriptsize $\pi_{\hat{l}+1}(B)$}

\put(5.5,1.4){\scriptsize $B$}
\put(2.5,3.6){\scriptsize $B$}

\thicklines
\put(3,0.5){\line(1,0){8}}
\put(1,2.75){\line(1,0){3}}

\put(3,1.3){\line(1,0){8}}
\put(1,3.5){\line(1,0){3}}

\end{picture}
\caption{Arc $B$ in projections $\pi_{\hat{l}+1}$ and $\pi_{\hat{l}}$.}
\label{fig:pi(B)}
\end{figure}
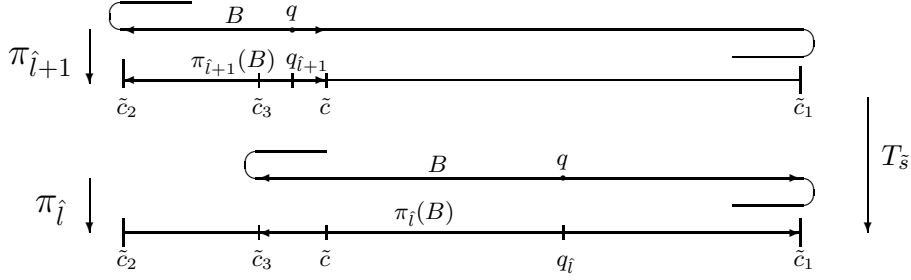

Let $Q\subset h(\Rr)$ be the closest (in the arc-length distance) arc to $q$ such that $\pi_{\hat{l}}:Q\rightarrow [\tilde{c}_2,\tilde{c}_1]$ is a bijection, 
\ie $Q$ has $\hat{l}$-pattern $12$. It follows that $q\notin Q$. 
For the rest of this section we abbreviate $T:=T_{\tilde{s}}$.
 
\begin{lemma}\label{lem:Qint}
Assume that an arc $Q \subset h(\Rr)$ has $\hat l$-pattern $12$.
If $Q\subset B_j$ for $j\in \N$ minimal, 
then $Q\subset (n_j,n_{j+2})$. 
\end{lemma}

\begin{proof}
Let us assume by contradiction that $n_j\in \text{int}(Q)$.
Note that $\pi_{\hat{l}}(Q)=[\tilde{c}_2,\tilde{c}_1]$ and let $\delta$ be 
chosen as in \eqref{eq:delta}.
In Proposition~\ref{prop:3.6} we obtain $\eps=\eps(\delta)$ such that $\eps\in (0,\delta)$ and thus it follows that $\eps s^{-(\hat{l}-l)}<\delta$ for every $\hat{l}\geq l$.
Note that $q\notin Q$, but $q\in B_j$.
We distinguish different cases for the position of $\pi_{\hat l}(n_j)$:

\textbf{Case I.} Assume that $|\pi_{\hat{l}}(n_j)-\tilde{c}_2|, |\pi_{\hat{l}}(n_j)-\tilde{c}|, |\pi_{\hat{l}}(n_j)-\tilde{c}_1|\geq \delta$.

Let
$$
[a,b]:=
\begin{cases}
[\tilde{c}_2,\pi_{\hat{l}}(n_j)+|\tilde{c}_2-\pi_{\hat{l}}(n_j)|], & \text{ if } |\tilde{c}_2-\pi_{\hat{l}}(n_j)|\leq|\tilde{c}_1-\pi_{\hat{l}}(n_j)|,\\ 
[\pi_{\hat{l}}(n_j)-|\tilde{c}_1-\pi_{\hat{l}}(n_j)|,\tilde{c}_1], & \text{ if } |\tilde{c}_2-\pi_{\hat{l}}(n_j)|>|\tilde{c}_1-\pi_{\hat{l}}(n_j)|,
\end{cases}
$$

and note that there exists an arc $Q'\subset Q$ such that $\pi_{\hat{l}}(Q')=[a,b]$.

First assume that $\tilde{c}\in [a,b]$ and $|a-\tilde{c}|,|b-\tilde{c}|\geq \delta$. Since we also assumed that $|\pi_{\hat l}(n_j)-\tilde c|\geq\delta$,
we can use Proposition~\ref{prop:eps-symmetric} for the interval $[a,b]$ to obtain that $T^{\hat{l}-l}|_{[a,b]}$ is not $\eps$-symmetric around $\pi_{\hat{l}}(n_j)$. But this contradicts that $B_j$ is $l$-link symmetric.

Now assume that either $|b-\tilde{c}|<\delta$ and $\tilde{c}\in [a,b]$ or $\tilde{c}\notin [a,b]$ and $a=\tilde{c}_2$. Let us study $T^{-2}([a,b])$.
Because we restrict $T$ to $[\tilde{c}_2,\tilde{c}_1]$, it follows that $T^{-2}(a)=T^{-2}(\tilde{c}_2)=\tilde{c}$. 
\newline
Assume first that  $|\pi_{\hat{l}+2}(n_j)-\tilde{c}|\geq\delta$. Let us set $\eta:=|\tilde{c}-\pi_{\hat{l}+2}(n_j)|+\delta$. Observe that there exists an arc $Q''\subset Q$ 
such that $\pi_{\hat l+2}(Q'')=[\pi_{\hat{l}+2}(n_j)-\eta, \pi_{\hat{l}+2}(n_j)+\eta]$. For the interval $\pi_{\hat l+2}(Q'')$ we can again apply
Proposition~\ref{prop:eps-symmetric} and obtain a contradiction.
\newline
If $|\pi_{\hat{l}+2}(n_j)-\tilde{c}|<\delta$ we proceed as in Case II or Case III.

If either $|a-\tilde{c}|<\delta$ and $\tilde{c}\in [a,b]$ or $\tilde{c}\notin [a,b]$ and $b=\tilde{c}_{1}$ we study $T^{-1}([a,b])$ and proceed analogously as in the preceding paragraph.

\textbf{Case II.} Let $\eps s^{-(\hat{l}-l)}<|\pi_{\hat{l}}(n_j)-\tilde{c}|<\delta$.

Let $u\in Q$ be such that $\pi_{\hat{l}}(u)=\tilde{c}$. Note that both $u$  and $n_j$ are centres of $l$-link symmetry. 
Denote by $x_0:=\tilde{c}$ and by $x_{1}:=\pi_{\hat{l}}(n_j)$. 
First set $x_{-1}\in [\tilde{c}_2,\tilde{c}_1]$ to be a reflection of $x_{1}$ over $x_{0}$, \ie $x_{-1}:=R_{x_0}(x_1)$. Continue inductively with $x_{k+1}:=R_{x_k}(x_{k-1})$ and $x_{-(k+1)}:=R_{x_{-k}}(x_{-(k-1)})$ for every $1\leq k \leq N-1$ where every $x_{k}\in [\tilde{c}_2,\tilde{c}_1]$ and $N\in \N$ is the smallest number such that $|x_{0}-x_{N}|>4\delta$.
Then it also follows that $|x_{0}-x_{-N}|>4\delta$. Because we reflect $\pi_{\hat{l}}$-projections of centres of $l$-link symmetry over the $\pi_{\hat{l}}$-projections of centres of $l$-link symmetry we obtain new $\pi_{\hat{l}}$-projections of centres of $l$-link symmetry. Thus we can find natural numbers $M,M',M''<N$ such that $|x_M-\tilde{c}|,|x_{-M'}-\tilde{c}|\geq \delta$ and $|x_{-M'}-x_{M}|=|x_{M''}-x_{M}|$. The interval $[x_{-M'},x_{M''}]$ is symmetric around $x_{M}$ and satisfies conditions from Proposition~\ref{prop:eps-symmetric} so we again obtain a contradiction with $B_j$ being an $l$-link symmetric arc.

\textbf{Case III.} Let $|\pi_{\hat{l}}(n_j)-\tilde{c}|\leq \eps s^{-(\hat{l}-l)}$. 

First we see that $Q$ has $\hat{l}-1$ pattern $312$, \ie $T|_Q$
maps in two branches on intervals $[\tilde{c}_3,\tilde{c}_1]$ and $[\tilde{c}_2,\tilde{c}_1]$. 
Let $\zeta:=\text{min}\{|\tilde{c}_2-\tilde{c}_3|,|\tilde{c}_3-\tilde{c}_1| \}$ and $J:=[\tilde{c}_3-\zeta, \tilde{c}_3+\zeta]\subset [\tilde{c}_2,\tilde{c}_1]$, as in Figure~\ref{fig:312}.

\begin{figure}[ht]
\unitlength=9mm
\begin{picture}(10,4)(1.5,-1)

\put(10,2.5){\line(-1,0){5}}
\put(13,2){\line(-1,0){8}} \put(5, 2.25){\oval(0.5,0.5)[l]}
\put(1,1.5){\line(1,0){12}} \put(13, 1.75){\oval(0.5,0.5)[r]}
\put(8,1){\line(-1,0){7}} \put(1, 1.25){\oval(0.5,0.5)[l]}

\put(1,0){\line(1,0){12}}
\put(1,-0,25){\line(0,1){0,5}}
\put(13,-0,25){\line(0,1){0,5}}
\put(5,-0,2){\line(0,1){0,4}}
\put(9,-0,2){\line(0,1){0,4}}

\put(-0.8,0.5){\large $\pi_{\hat{l}-1}$}
\put(0.5, 1){\vector(0, -1){1}}
\put(1,-0,5){\small $\tilde{c}_2=\tilde{c}_3-\zeta$}
\put(5,-0,5){\small $\tilde{c}_3$}
\put(13,-0,5){\small $\tilde{c}_1$}
\put(9,-0,5){\small $\tilde{c}_3+\zeta$}
\put(7,0.25){\large $J$}
\linethickness{0.5mm}
\put(1,0){\line(1,0){8}}

\end{picture}
\caption{Interval $J$ as in Case III.}
\label{fig:312}
\end{figure}
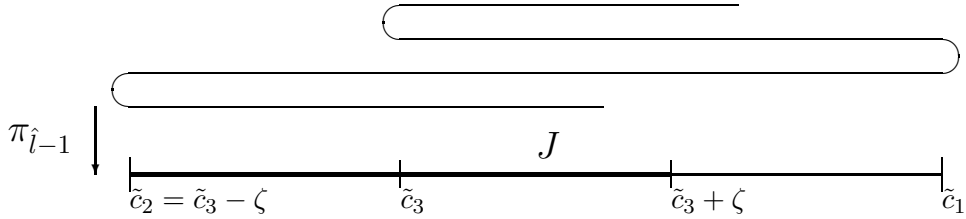

Because $|\tilde{c}_2-\tilde{c}_3|>\delta$ we can use
Corollary~\ref{cor:2} for interval $J$ and obtain that 
$T^{\hat{l}-l-1}|_{J}$ is not $\eps$-symmetric and this again contradicts
that $B_j$ is $l$-link symmetric.

\textbf{Case IV.} Let $\eps s^{-(\hat{l}-l)}<|\pi_{\hat{l}}(n_j)-\tilde{c}_2|<\delta$ (the case $\eps s^{-(\hat{l}-l)}<|\pi_{\hat{l}}(n_j)-\tilde{c}_1|<\delta$ goes similarly).

We obtain that $\eps s^{-(\hat{l}-l+2)}<|\pi_{\hat{l}+2}(n_j)-\tilde{c}| <\delta s^{-2}<\delta$ and so we proceed as in Case II.

\textbf{Case V.} Let $ |\pi_{\hat{l}}(n_j)-\tilde{c}_2|\leq\eps s^{-(\hat{l}-l)}$ (the case $|\pi_{\hat{l}}(n_j)-\tilde{c}_1|\leq \eps s^{-(\hat{l}-l)}$ goes similarly).

We obtain that $|\pi_{\hat{l}+2}(n_j)-\tilde{c}|\leq \eps s^{-(\hat{l}-l+2)} $ and proceed as in Case III.
\end{proof}

The following lemma strengthens Proposition~\ref{prop:eps-symmetric},
in the sense that given $\hat l > l$ and an arc $Q \subset \Rr$
with $\hat l$-pattern $12$, Lemma~\ref{lem:12} implies that
if an arc $S\subset \Rr$
has the same $l$-pattern (or reverse $l$-patterns) as $Q$,
then $S$ itself must have $\hat l$-pattern $12$.

\begin{lemma}\label{lem:12}
Assume that $\tilde{c}$ is not recurrent.
Then there is $\eps > 0$ such that, whenever
$T^i|_{[\tilde{c}_2, \tilde{c}_1]}$ and $T^j|_{[a,b]}$ are $\eps$-close
for some interval $[a,b]\subset[\tilde{c}_2,\tilde{c}_1]$, there
is $k \ge 0$ and a closed interval
$J:=[a',b']\subset[\tilde{c}_2,\tilde{c}_1]$ such that
$|a'-a|,|b'-b|<\eps$
so that $T^k$ maps $J$ homeomorphically onto $[\tilde{c}_2, \tilde{c}_1]$.
\end{lemma}

\begin{remark}
The closed interval $J$ addresses the technicality that if e.g.\ $i=j=0$
and $a = \tilde{c}_2+\eps/2$, $b = \tilde{c}_1-\eps/2$,
then $T^i|_{[\tilde{c}_2, \tilde{c}_1]}$ and $T^j|_{[a,b]}$ are $\eps$-close,
but without the adjustment of $J = [\tilde{c}_2, \tilde{c}_1]$, the lemma
would fail.
\end{remark}

\begin{proof}
Take $\delta$ as in \eqref{eq:delta} and $\eps < \delta/10$.
Assume that $T^i|_{[\tilde{c}_2, \tilde{c}_1]}$ and $T^j|_{[a,b]}$ are
$\eps$-close,
with homeomorphism $h:J \to [\tilde{c}_2,\tilde{c}_1]$ as in
Definition~\ref{def:eps-close}.

If $i > j$, then $T^i|_{[\tilde{c}_2,\tilde{c}_1]}$ has more branches than
$T^j|_{[a,b]}$,
so they cannot be $\eps$-close.
If $i = j$, then there is nothing to prove.
Therefore we can assume that $i < j$ and take $k = j-i$.

Suppose that $T^k|_{[a',b']}$ is homeomorphism onto a subinterval of $[\tilde{c}_2,\tilde{c}_1]$.
If $T^k([a', b']) \supset [\tilde c_2+\eps \tilde s^{-i}, \tilde c_1-\eps \tilde s^{-i}]$, then (since by non-recurrence 
$\tilde c_r$ cannot be $\eps \tilde s^{-i}$-close to $\tilde c_2$ or $\tilde c_1$ for every $r>2$),
we can adjust $[a', b']$ so that
$T^k([a', b']) \supset [\tilde c_2, \tilde c_1]$.
In this case, the lemma is proved.
If on the other hand  $T^k([a', b']) \not\supset [\tilde c_2+\eps \tilde s^{-i}, \tilde c_1-\eps \tilde s^{-i}]$  for every $a',b'$ satisfying assumptions of the lemma, then $T^j|_{[a,b]}$ cannot be $\eps$-close to $T^i|_{[\tilde c_2, \tilde c_1]}$. 

Since $T^k([a', b']) \subset [\tilde c_2, \tilde c_1]$, there
is $t \in J$ such that $x := h(t) = T^k(t)$; let $U \owns t$ be
the maximal closed interval in $[a', b']$ such that $T^k|_U$ is monotone. 

Now suppose that $T^k|_{[a',b']}$ is not a homeomorphism, and
take $t' \in \partial U \setminus \{ a',b'\}$ closest to $t$, so that $T^k(t') = \tilde{c}_r$ for some $r \ge 1$. Let  $U'$ be
the maximal neighbourhood of $t'$ such that $T^k(U')$
is contained in a $\delta$-neighbourhood $V$
of $\tilde{c}_r$. It follows that $T^j|_{U'}$ is $\eps$-symmetric (see
Figure~\ref{fig:lem12}).

\begin{figure}[ht]
\unitlength=9mm
\begin{picture}(10,4)(2,-1)
\put(10,2.7){\line(-1,0){3}}
\put(5,2.2){\line(1,0){5}} \put(10, 2.45){\oval(0.5,0.5)[r]}

\put(5.5,0.25){\vector(0,1){1.7}}\put(4.9,0.7){$T^k$}
\put(6,0.25){\vector(0,1){1}}\put(6.2,0.6){$h$}

\put(1,1.5){\line(1,0){12}}
\put(1,1.375){\line(0,1){0,25}}
\put(13,1.375){\line(0,1){0,25}}
\put(9,2.1){\line(0,1){0,2}}
\put(5.7,-0.1){$[$} \put(5.6,-0.6){$a'$}
\put(12.2,-0.1){$]$} \put(12.1,-0.6){$b'$}
\put(9.5,-0.1){$($} \put(10.5,-0.1){$)$} \put(9.85,0.15){\small $U'$}
\put(9.5,1.4){$($} \put(10.5,1.4){$)$} \put(9.85,1.75){\small $V$}

\put(1,0){\line(1,0){12}}
\put(5.8,0.03){\line(1,0){4.2}}
\put(1,-0,125){\line(0,1){0,25}}
\put(13,-0,125){\line(0,1){0,25}}
\put(9,-0,1){\line(0,1){0,2}}
\put(10,1.4){\line(0,1){0,2}}
\put(9,1.4){\line(0,1){0,2}}
\put(10,-0.1){\line(0,1){0,2}}

\put(0,0){\large $\pi_{j}$}
\put(0,1.5){\large $\pi_{i}$}
\put(1,-0,5){\small $\tilde{c}_2$}
\put(13,-0,5){\small $\tilde{c}_1$}
\put(1,1){\small $\tilde{c}_2$}
\put(13,1){\small $\tilde{c}_1$}
\put(10,-0.5){\small $t'$}
\put(7.25,0.25){\small $U$}
\put(9,-0.5){\small $t$}
\put(10,1){\small $\tilde{c}_r$}
\put(7.9,1){\small $x=h(t)$}
\put(7.7,1.7){\small $x=T^k(t)$}
\end{picture}
\caption{Step in the proof of Lemma~\ref{lem:12}.}
\label{fig:lem12}
\end{figure}
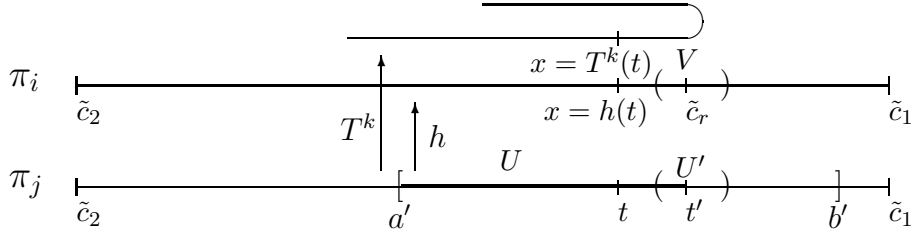

If $h|_U$ and $T^k|_U$ have the same orientation, then, by $\eps$-closeness,
$|T^i(h(u))- T^j(u)| < \eps$ for all $u \in U$, which means that
$|T^k(u)-h(u)| \le \eps \tilde s^{-i}$ for all $u \in U$.
However, $T^i|_V$ is not $\eps$-symmetric due to Corollary~\ref{cor:2},
and therefore the $\eps$-closeness is violated on the
neighbourhood $U'$.

On the other hand, if $h|_U$ and $T^k|_U$ have opposite orientation, then
$T^i$ is $\eps$-symmetric on a neighbourhood of $x$,
with $\tilde{c}_r$ in its closure.
Let $V'$ be the mirror image of $V$ when reflected in $x$.
Then by the $\eps$-symmetry of $T^i$ around $x$ and around $\tilde{c}_r$,
$T^i$ has to be $\eps$-symmetric on $V'$ as well.
But this contradicts Corollary~\ref{cor:2} again, completing the proof.
\end{proof}

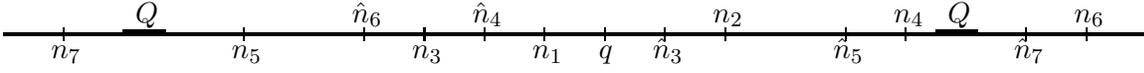
\begin{figure}[ht]
\unitlength=8mm
\begin{picture}(15,2)(1,0)
\thicklines
\put(-1,1){\line(1,0){19}}
\put(1,1.06){\line(1,0){0.7}}\put(1,1.03){\line(1,0){0.7}}\put(1.2,
1.2){\small $Q$}
\put(14.5,1.06){\line(1,0){0.7}}\put(14.5,1.03){\line(1,0){0.7}}\put(14.7,
1.2){\small $\hat{Q}$}
\thinlines
\put(9,0.9){\line(0,1){0.2}}\put(8.9, 0.6){\small $q$}
\put(8,0.9){\line(0,1){0.2}}\put(7.8, 0.6){\small $n_1$}
\put(6,0.9){\line(0,1){0.2}}\put(5.8, 0.6){\small $n_3$}
\put(10,0.9){\line(0,1){0.2}}\put(9.8, 0.6){\small $\hat n_3$}
\put(3,0.9){\line(0,1){0.2}}\put(2.8, 0.6){\small $n_5$}
\put(13,0.9){\line(0,1){0.2}}\put(12.8, 0.6){\small $\hat n_5$}
\put(0,0.9){\line(0,1){0.2}}\put(-0.2, 0.6){\small $n_7$}
\put(16,0.9){\line(0,1){0.2}}\put(15.8, 0.6){\small $\hat n_7$}
\put(11,0.9){\line(0,1){0.2}}\put(10.8, 1.2){\small $n_2$}
\put(7,0.9){\line(0,1){0.2}}\put(6.8, 1.2){\small $\hat n_4$}
\put(14,0.9){\line(0,1){0.2}}\put(13.8, 1.2){\small $n_4$}
\put(5,0.9){\line(0,1){0.2}}\put(4.8, 1.2){\small $\hat n_6$}
\put(17,0.9){\line(0,1){0.2}}\put(16.8, 1.2){\small $n_6$}
\end{picture}
\caption{The midpoints and endpoints of the arcs $B_i$.} \label{fig:order}
\end{figure}

In the proof of Theorem~\ref{thm:R} we take $\eps>0$ small enough such
that both Proposition~\ref{prop:eps-symmetric} and Lemma~\ref{lem:12}
apply.

\begin{proof}[Proof of Theorem~\ref{thm:R}.]
Let $Q \subset h(\Rr)$ be an arc in $h(\Rr)$ with $\hat l$-pattern
$12$.  As we already observed, $q\notin Q$.
Assume without loss of generality that $Q$ is the closest to $q$,
in the sense that $(r_{q}(Q), Q)$ contains no
other arc with $\hat l$-pattern $12$, where $r_{q}(Q)$ denotes the
reflection of the arc $Q$ over point $q$.

 Let $P$ be the $l$-pattern of $Q$; it is the $T^{\hat l - l}$-image of
the $\hat{l}$-pattern $12$.

Now let $j$ be minimal such that $B_j\cap\text{int}(Q)\neq \emptyset$.
Then $Q \subset (n_j,n_{j+2})$ by Lemma~\ref{lem:Qint}.
Since $B_j$ is $l$-link-symmetric around $n_j$, we can reflect $Q$ in $n_j$,
obtaining an arc $\hat Q\subset \mathfrak{R}$ with $l$-pattern $P$ (see
Figure~\ref{fig:order}). Lemma~\ref{lem:12} shows that this is impossible,
unless $\hat Q$ itself has an $\hat l$-pattern $12$,
contradicting the choice of $Q$. Thus there exists no arc $Q\subset h(\Rr)$
with $l$-pattern $P$ which contradicts Lemma~\ref{lem:dense}.
\end{proof}

\section{The Core Ingram Conjecture}\label{sec:CIC}

Recall that $p\in \CUIL$ is called a \emph{$k$-point} if there exists $n>0$ such that
$\pi_{k+n}(p)=c$, and we write $L_k(p) = n$.
Note that if $c$  is non-periodic, $k$-level $n$ is unique.

\begin{definition}
Let $\Uu\subset \CUIL$ be an arc-component and $x\in\Uu$. We say that a $k$-point $p\in\Uu$ such that $L_{k}(p)=n$ is a \emph{salient $k$-point with respect to $x$} if
$\pi_{k+n}|_{[x, p]}$ is injective.
\end{definition}

\begin{remark}
The above definition says that a salient point $p$ is
a $k$-point of level $n$ and that there are no $k$-points between $x$ and $p$ with greater $k$-level than $n$. In this sense,
it corresponds to the previous definition of salient point (for example in \cite{BBS}). Note that we will work with salient $k$-points with respect to $\rho, \tilde{\rho}$ or $q$
but because it is clear with respect to which point we work we refer to them 
only as salient $k$(or $l$)-points.
\end{remark}

\begin{lemma}
For any $i\in \N$, the midpoint $m_i$ of $A_i\subset \Rr$ is
a salient $k$-point with respect to $\rho$ and its $k$-level is $i$.
\end{lemma}
\begin{proof}
By the definition of $A_i$, we obtain that $\rho\in A_i$, $\pi_{k+i}(m_i)=c$ and $\pi_{k+i}|_{A_i}$ is injective in $[c_2,c_1]$. This proves the claim.
\end{proof}

We consider $l$-link symmetric arcs
$B_i=h(A_i)\subset\tilde{\Rr}$, where
the chains $\chain_k$ and $\tilde{\chain_{l}}$ satisfy \eqref{eq:refine}.
Let $\tilde{A_i}\subset \tilde{\Rr}\subset \CUILt$ be the arc-component of $\pi^{-1}_{l+i}([\tilde{c}_2, \hat{\tilde{c}}_2])$ containing $\tilde\rho$ for every $i\in \N$.
Arcs $A_i$ are all $l$-symmetric with salient $l$-points $\tilde{m}_i$ of level $i$ and they form a complete sequence with respect to $\tilde\rho=(\ldots,\tilde{r},\tilde{r})$.

Let $n_i$ and $\tilde{m}_i$ be the midpoints of the arcs $B_i$ and $\tilde{A}_i$ respectively. In the next two lemmas we show how $B_i$ and $\tilde{A}_i$ relate to each other.

\begin{lemma}\label{lem:ms}
There exists $N\in \N$ such that for every $j\geq N$ there exists $j'\in \N$ such that
$\tilde\rho\in B_j$, $q=h(\rho)\in\tilde{A}_{j'}$ and $n_j=\tilde{m}_{j'}\notin[q,\tilde\rho]$, where $[q,\tilde\rho]$ denotes the shortest (in arc-length) arc in $\tilde{\Rr}$ containing $q$ and
$\tilde\rho$.
\end{lemma}

\begin{proof}
By Lemma~\ref{l2} and applying $h$ we obtain that 
$\cup_{i \textrm{ odd}}B_i=\cup_{i\textrm{ even}}B_i=\tilde{\Rr}$ and $B_i\subset B_{i+2}$ for every $i\in \N$, 
so there exists $N$ such that $[\tilde \rho, q] \subset B_j$ for all 
$j \ge N$. By Lemma~\ref{lem:mid} it follows that $n_{i+2} \in \partial B_i$.
This implies that $n_j \notin [\tilde \rho, q]$ for $j \ge N+2$.
The argument for the arcs $\tilde A_i$ is analogous.
\end{proof}

\begin{lemma}
There exists $N\in\N$ such that for every $j\geq N$ there exists $j'\in\N$ such that $B_j=\tilde{A}_{j'}$, up to link-symmetry.
\end{lemma}

\begin{proof}
Take $N$ from Lemma~\ref{lem:ms}. 
Assume by contradiction that there exists $j\geq N$ such that $B_j\neq \tilde{A}_{j'}$ for every $j'\in\N$.
By  completeness of $\{\tilde{A}_i\}_{i\in \N}$, there exists some $j'\in\N$ such that $n_j=\tilde{m}_{j'}$. As $B_j$ and $\tilde{A}_{j'}$ are both $l$-link symmetric with the same midpoint,
either $B_j\subsetneq\tilde{A}_{j'}$ or $\tilde{A}_{j'}\subsetneq B_j$. Assume that $B_j\subsetneq\tilde{A}_{j'}$. Note that since  $n_j=\tilde{m}_{j'} \notin
[\tilde\rho, q]$ we obtain that $n_{j+2}\in(\tilde{m}_{j'}, \tilde{m}_{j'+2})$. But this is impossible since $\tilde{m}_i\notin(\tilde{m}_{j'}, \tilde{m}_{j'+2})$
for every $i\in \N$. The second case follows similarly, but instead of the completeness of $\{\tilde{A}_i\}_{i\in \N}$ we use the completeness of sequence $\{B_i\}_{i\in\N}$.
\end{proof}

\begin{proposition}
There exist $N, M\in\N$ such that $L_l(n_{N+i})=i+M$ for every $i\in\N_0$.
\end{proposition}

\begin{proof}
Take $N$ from Lemma~\ref{lem:ms}. There exist $j', j''\in \N_0$ such that
(up to link-symmetry):
$$B_N=\tilde{A}_{j'}, B_{N+2}=\tilde{A}_{j'+2}, B_{N+4}=\tilde{A}_{j'+4},\ldots$$
$$B_{N+1}=\tilde{A}_{j''}, B_{N+3}=\tilde{A}_{j''+2}, B_{N+5}=\tilde{A}_{j''+4},\ldots$$
or in terms of $l$-levels $L_l(n_{N+2i})=j'+2i$, $L_l(n_{N+2i+1})=j''+2i$ for all $i\in \N_0$.
So far we only know that $j'$ and $j''$ must be of different parity. Assume $j''>j'$, so there exists an odd $j\geq 1$ such that $j''=j'+j$.
Since $B_N=h(A_N)\not\subset h(A_{N+1})=B_{N+1}$, from Lemma~\ref{lem:kappa} we conclude that $j<\kappa$. Assume that $j>1$ and take $i = \kappa-j$.
From Lemma~\ref{lem:kappa} we obtain that $\tilde{A}_{j'}\subseteq \tilde{A}_{j'+\kappa}$. But $\tilde{A}_{j'}=B_N$, $\tilde{A}_{j'+\kappa}=\tilde{A}_{j'+j+i}=\tilde{A}_{j''+i}=B_{N+i+1}$.
Thus we get $h(A_N)=B_N\subseteq B_{N+i+1}=h(A_{N+i+1})$ which is a contradiction because $i+1<\kappa$ and $i+1$ odd. We conclude that $j=1$. 
\\
The other possibility is that $j''<j'$. Since also $B_{N+1}=h(A_{N+1})\not\subseteq h(A_N)=B_N$, we conclude that $j'=j''+j$, where $j<\kappa$ odd. Recall that
$\tilde{A}_{j''}\subseteq \tilde{A}_{j''+\kappa}$, but $\tilde{A}_{j''}=B_{N+1}$ and $\tilde{A}_{j''+\kappa}=B_{N+\kappa-j}$, where $\kappa-j\in \N$ is even. This gives the existence of
an even $0<i<\kappa$ such that $B_{N+1}\subseteq B_{N+i}$, which is again a contradiction.
\\
So the only possibility is $j''=j'+1$, which gives $B_{N+i}=\tilde{A}_{j'+i}$
up to link-symmetry for every $i\in\N_0$ and this finishes the proof.
\end{proof}

\noindent
So far we have shown that there exist $N, M\in \N$ such that $h$ maps
the salient point of $k$-level $i+N$ \emph{close to} the salient point of $l$-level $i+M$ for every $i\in \N_0$. Here \emph{close to} means that $h(m_{i+N})$ is in the same link of $\tilde{\chain_l}$
as $\tilde{m}_{i+M}$ and the arc-component of that link containing point $\tilde{m}_{i+M}$ also contains the point $h(m_{i+N})$. Note that this works for any $k$ and $l$ such that
$h(\chain_k)\preceq \tilde{\chain}_l$.
The salient $(k+N)$-point of $k$-level $i$ is the salient $k$-point of $k$-level $i+N$. Therefore, if we consider $\chain_{k+N}$ instead of $\chain_k$, 
then $h(m_i)$ is close to $\tilde{m}_{i+M}$ for every $i\geq 1$.

The proof of the Core Ingram Conjecture now follows analogously as in
\cite{FL}. We first need to prove that $h$ preserves the sequence of $k$-points and then argue that sequences of $k$-points and $l$-points
of $\Rr$ and $\tilde{\Rr}$ respectively are never the same, unless $s = \tilde s$.

\begin{proposition}\label{prop}
Let $n\in \N$ and $x\in\Rr$ be a $k$-point with $k$-level $n$. Then $h(x)\in \tilde{\Rr}$ is in the link of $\tilde{\chain}_l$ that contains $\tilde{m}_{n+M}$ and
the arc-component of the link
that contains $h(x)$ also contains an $l$-point $y$ with $l$-level $n+M$ (see Figure~\ref{fig:map}).
\end{proposition}

\begin{figure}[ht]
\unitlength=16mm
\begin{picture}(1.5,2)(5.8,0.5)
\put(2,2){\line(1,0){2}} \put(4, 1.5){\oval(1,1)[r]}
\put(2,1){\line(1,0){2}}
\put(4, 1.5){\oval(1.5,2)[l]}
\put(4, 1.5){\oval(2,2)[r]}
\put(4.2, 1.4){\small $x$}
\put(4.5, 1.5){\circle*{0.05}}
\put(5.5, 1.5){\vector(1, 0){1.7}}
\put(6.3, 1.8){\small $h$}

\put(7.6, 2.1){\line(1, 0){2}}
\put(9.5, 1.7){\line(1,0){0.5}}
\put(9.5,1.3){\line(1, 0){0.5}}
\put(7.6, 0.9){\line(1, 0){2}}

\put(7.6, 2.2){\line(1, 0){2}}
\put(9.5, 1.6){\line(1,0){0.5}}
\put(9.5,1.4){\line(1, 0){0.5}}
\put(7.6, 0.8){\line(1, 0){2}}

\put(9.5, 1.9){\oval(1.7, 0.6)[r]}
\put(9.5, 1.1){\oval(1.7, 0.6)[r]}
\put(10, 1.5){\oval(1.3, 0.2)[l]}
\put(8.65,1.45){\scriptsize $\tilde{m}_{n+M}$}
\put(9.25,1.5){\circle*{0.05}}

\put(9.5, 1.9){\oval(1.5, 0.4)[r]}
\put(9.5, 1.1){\oval(1.5, 0.4)[r]}
\put(10, 1.5){\oval(1.5, 0.4)[l]}
\put(9.7, 1.5){\oval(2.2,2)}
\put(9.8,2.2){\circle*{0.05}}
\put(9.8,2.25){\small $h(x)$}
\put(9.5,1.45){\small $y$}
\put(9.35,1.5){\circle*{0.05}}

\end{picture}
\caption{Claim of the Proposition~\ref{prop}.}
\label{fig:map}
\end{figure}
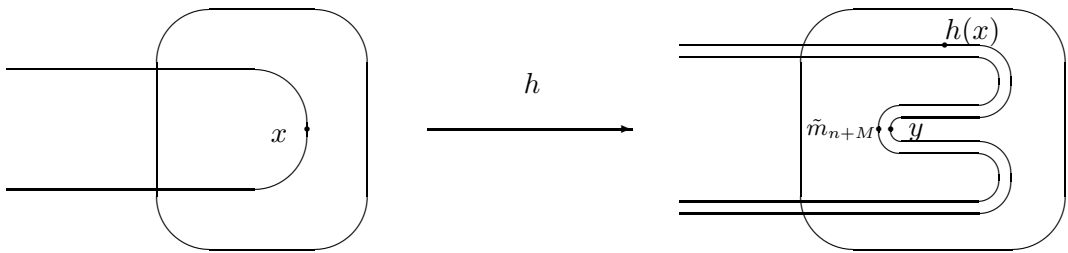

\begin{proof}
For $i\in\N$ denote by $S_i$ the longest arc in $\Rr$ containing $m_i$ such that $\pi_{k+i}|_{S_i}$ is injective. Note that $S_i$ is exactly the arc-component of $\pi^{-1}_{k+i}([c_2, c_1])$
which contains $m_i$ and that $\pi_{k+i}(S_i)=[c_2, c_1]$, because we are on arc-component $\Rr$.
 Also note that $A_i\subset S_i$ and the endpoint of $A_i$ projecting with $\pi_{k+i}$ to $c_2$
 agrees with one endpoint of $S_i$.
Let $S_i^{\rho}, S_i^{\neg\rho} \subset S_i$ be the arc-components of $\pi^{-1}_{k+i}([c, c_1])$ and of $\pi^{-1}_{k+i}([c_2, c])$ respectively, with $m_i$
as the common boundary point.
Note that $\rho\in S_i^{\rho}$ and its endpoints are $m_i$ and $m_{i+1}$. Also, $\rho\notin S_i^{\neg\rho}$ and its endpoints are $m_i$ and $m_{i+2}$.
Also note that $S_i^{\neg\rho}$ is shorter (in arc-length) than $S_i^{\rho}$ and that $S_{i+1}^{\rho}=S_i$. We will prove the proposition for
$k$-points in $S_i$ by induction on $i$.\\
Note that all $k$-points in $S_1$ are salient, and by the remarks preceding this proposition it follows that the proposition holds for salient points.
Assume that the proposition holds for all $k$-points in $S_i(=S_{i+1}^{\rho}).$ Take a $k$-point $x\in S_{i+1}^{\neg\rho}\setminus\{m_{i+1}, m_{i+3}\}$ with $k$-level $n$.
Note that $n<i+1$ by the definition of $S_{i+1}$. Also, since $S_{i+1}^{\neg\rho}$ is shorter than $S_{i+1}^{\rho}$ there exists a $k$-point $\hat{x}\in S_{i+1}^{\rho}$
such that $[x, \hat{x}]$ is $k$-symmetric with midpoint $m_{i+1}$. Observe that $h([x, \hat{x}])$ is $l$-link symmetric with midpoint $\tilde{m}_{i+1+M}$, because it is the point with
the highest $l$-level
in the link containing $h(m_{i+1})$. Since $\hat{x}\in S_{i+1}^{\rho}=S_i$, $h(\hat{x})$ is in the link containing $\tilde{m}_{n+M}$ and the arc-component
of the link containing $h(\hat{x})$ contains
$l$-point $\hat{y}$ such that $L_l(\hat{y})=n+M$. Take such $\hat{y}$ closest (in arc-length) to $\tilde{m}_{i+1+M}$ such that there are no points of $l$-level greater or equal than
$i+1+M$ in $(\tilde{m}_{i+1+M}, \hat{y})$. Since $n<i+1$,
we obtain that $L_l(\hat{y})=n+M<i+1+M=L_l(\tilde{m}_{i+1+M})$.
Note that $h(x)\in(\tilde{m}_{i+3+M}, \tilde{m}_{i+1+M})$ and thus there must exist an $l$-point $y$ such that the arc $[y, \hat{y}]$ is $l$-symmetric with midpoint
$\tilde{m}_{i+1+M}$. This implies that $y$ and $\hat{y}$ both have the same level $n+M$, and that they belong to the same link. The arc-component of the link containing $y$
must also contain point $h(x)$.
This concludes the proof for every $k$-point in $S_{i+1}$. Since $\cup_{i}S_i=\Rr$,
this concludes the proof.  
\end{proof}

\begin{proposition}\label{prop:bijec}
Let  $k, l, \hat k$ be such that $h(\chain_k)\preceq\tilde{\chain}_l\preceq h(\chain_{\hat{k}})$ holds as in \eqref{eq:refine}.
Take $M, M'\in\N$ such that $h$ maps every $k$-point with $k$-level $n$ close to 
$l$-point with $l$-level $n+M$ and $h^{-1}$ maps every $l$-point with $l$-level $n$ close to $\hat{k}$-point with $\hat{k}$-level $n+M'$.
Then for every $K \in \N$, there is an orientation preserving bijection between
$$
\{ u \in [m_K, m_{K+1}] : L_k(u) = n \}
\text{ and }
\{ v \in [\tilde m_{K+M}, \tilde m_{K+1+M}] : L_l(v) = n+M \}.
$$
\end{proposition}

\begin{proof}
First we claim that $M+M'=k-\hat{k}$. Take the salient $k$-point 
$m_i$ with $L_k(m_i) = i$ and note that it is also a salient $\hat{k}$-point 
with $L_{\hat k}(m_i) = i+k-\hat k$.
Note that by remarks before Proposition~\ref{prop}, homeomorphism $h$ maps the salient $k$-point with level $i$ close to the salient $l$-point with $l$-level $i+M$, which is mapped by $h^{-1}$ close to the salient
$\hat{k}$-point with $\hat{k}$-level
$i+M+M'$. This means that the salient $\hat{k}$-point with $\hat{k}$-level $i+k-\hat{k}$ belongs to the same arc-component of the same link of the chain $\chain_{\hat{k}}$ that contains
the salient $\hat{k}$-point with $\hat{k}$-level $i+M+M'$. But this is only possible if the points are equal which implies that $M+M'=k-\hat{k}$.
\\
Denote by $z_i$, $i=1, \ldots, a$, all $k$-points with $k$-level $n$ in $[m_K, m_{K+1}]$ such that $m_K\prec z_1\prec\cdots\prec z_a\prec m_{K+1}$ (where $x\prec y\prec z$ if $[x, y]\subset [x, z]$).
Similarly, denote by $\tilde{z}_j$, $j=1, \ldots, b$, all $l$-points with $l$-level $n+M$ in $[\tilde{m}_{K+M}, \tilde{m}_{K+1+M}]$ such that
$\tilde{m}_{K+M}\prec\tilde{z}_1\prec\cdots\prec\tilde{z}_b\prec\tilde{m}_{K+1+M}$. We will first prove that $a\leq b$.
\\
 Recall that for an $l$-point $u$ such that $u\in\ell\in\tilde{\chain}_l$ we denote the arc-component of $u$ in $\ell$ by $A^u$.
 We can find $N>0$ such that $A^{\sigma^N(\tilde{m}_{K+M})}, A^{\sigma^N(\tilde{z}_1)},\ldots,A^{\sigma^N(\tilde{z}_b)}, A^{\sigma^N(\tilde{m}_{K+1+M})}$ are all different.
Also, every point $u\in\{\sigma^N(\tilde{m}_{K+M}), \sigma^N(\tilde{z}_1),\ldots,\sigma^N(\tilde{z}_b),
\sigma^N(\tilde{m}_{K+1+M})\}$ has to be a midpoint of  $A^{u}$.
Otherwise, there would exist another $l$-point with $l$-level $n+M+N$ in the same arc-component which is impossible since we separated them. Since $\sigma^N(m_K)=m_{K+N}$
and $\sigma^N(\tilde{m}_{K+M})=\tilde{m}_{K+M+N}$, we get from Proposition~$\ref{prop}$ that for every $i\in\{1, \ldots, a\}$ there exists unique $j\in\{1, \ldots, b\}$ such
that $h(\sigma^N(z_i))\in A^{\sigma^N(\tilde{z}_j)}$. This defines a function $x\mapsto\tilde{x}$ for every $k$-point $x\in[m_K, m_{K+1}]$ with $L_k(x)=n$. Note that we can
take $N$ such that $\sigma^{N}$ preserves orientation and so $x\prec y$ implies $\tilde{x}\prec\tilde{y}$.
\\
Next we want to prove that $x\mapsto\tilde{x}$ is injective.
Assume there are $i_1, i_2\in\{1, \ldots, a\}$ such that
$h(\sigma^N(z_{i_1})), h(\sigma^N(z_{i_2}))\in A^{\sigma^N(\tilde{z}_j)}$, for some $j\in\{1, \ldots, b\}$.
There exists a $k$-point $w$ such that $\sigma^N(z_{i_1})\prec w\prec\sigma^N(z_{i_2})$ and such that $L_k(w) > n+N$.
Note that $h(w)\in A^{\sigma^N(\tilde{z}_j)}$. But then there exists an $l$-point $\tilde{w}\in A^{\sigma^N(\tilde{z}_j)}$ with $l$-level strictly greater than $n+N+M$ which is
in contradiction with $\sigma^N(\tilde{z}_j)$ being the center of the link. This proves that the above function $x\mapsto\tilde{x}$ is injective, \ie $a\leq b$.
\\
It follows that
\begin{align*}
\#\{ & k\text{-points in $[m_K, m_{K+1}]$ with $k$-level $n$}\} \\
& \le
\#\{ l\text{-points in $[\tilde{m}_{K+M}, \tilde{m}_{K+1+M}]$ with $l$-level $n+M$}\}\\
& \le
\#\{ \hat{k}\text{-points in $[m_{K+M+M'}, m_{K+1+M+M'}]$ with $\hat{k}$-level $n+M+M'$}\}.
\end{align*}
We proved that $M+M'=k-\hat{k}$ so the last number is
equal to the number of $\hat{k}$-points in $[m_{K+k-\hat{k}}, m_{K+1+k-\hat{k}}]$ with $\hat{k}$-level $n+k-\hat{k}$. But this is actually equal to the number of $k$-points in
$[m_K, m_{K+1}]$ with $k$-level $n$. This proves that $a=b$.
\end{proof}

\begin{proof}[Proof of Theorem \ref{CIC}.]

We claim that the $k$-pattern of $[m_{n-1}, m_n]$ is equal to the $(l+M)$-pattern of $[\tilde{m}_{n-1}, \tilde{m}_n]$ and that $T^n_s(c) > c$ if and only if $T^n_{\tilde s}(c) > c$ for every $n\geq 2$. This gives $s=\tilde{s}$.

The claim is obviously true for $n=2$. 
For the inductive step, assume that it is true for all positive integers $<n$. 

Specifically, the $k$-pattern of $[m_{n-2}, m_{n-1}]$ is the $(l+M)$-pattern of $[\tilde{m}_{n-2}, \tilde{m}_{n-1}]$. Denote all $k$-points in $[m_{n-2}, m_{n-1}]$ by $m_{n-2}=x_0\prec x_1\prec\ldots\prec x_i\prec x_{i+1}=m_{n-1}$. Denote all $(l+M)$-points in $[\tilde{m}_{n-2}, \tilde{m}_{n-1}]$ analogously by $\tilde{m}_{n-2}=x_0\prec \tilde{x}_1\prec\ldots\prec \tilde{x}_i\prec \tilde{x}_{i+1}=\tilde{m}_{n-1}$. Since patterns are the same, $L_k(x_j)=L_{l+M}(\tilde{x}_j)$ for all $j\in\{0, \ldots, i+1\}$. By the inductive assumption it follows that $c\in\pi_k(x_j, x_{j+1})$ if and only if $c\in\pi_{l+M}(\tilde{x}_j, \tilde{x}_{j+1})$  for all $j=0, \ldots, i$. Since $\sigma([m_{n-2}, m_{n-1}])=[m_{n-1}, m_n]$ and every subarc $[x_j, x_{j+1}]$ is mapped to the subarc $[\sigma(x_j), \sigma({x}_{j+1})]$ with $k$-pattern $L_k(x_j)+1, 1, L_k(x_{j+1})+1$ or $L_k(x_j)+1, L_k(x_{j+1})+1$ according to whether $\pi_k([x_{j-1}, x_{j}])$ contains $c$ or not, inductive hypothesis for $n-1$ completely determines the $k$-pattern of $[m_{n-1}, m_n]$. The same holds for the arc $[\tilde{m}_{n-2}, \tilde{m}_{n-1}]$. Since we assumed that $T^{n'}_s(c)<c$ if
and only if $T^{n'}_{\tilde{s}}(c)<c$ for all $n'<n$, this gives that the $k$-pattern of $[m_{n-1}, m_n]$ is the same as the $(l+M)$-pattern of $[\tilde{m}_{n-1}, \tilde{m}_n]$.

From now on we study $[m_{n-1}, m_n]$.
Write $m_{n-1}\prec x_1\prec\cdots\prec x_i\prec m_n$ and
$\tilde{m}_{n-1}\prec\tilde{x}_1\prec\cdots\prec\tilde{x}_i\prec \tilde{m}_n$, where $\{x_1, \ldots, x_i\}$ is the set of all $k$-points in $[m_{n-1}, m_n]$ and $\{\tilde{x}_1, \ldots, \tilde{x}_i\}$ is the set of all $(l+M)$-points in $[\tilde{m}_{n-1}, \tilde{m}_n]$. From the previous paragraph we obtain that $L_k(x_j)=L_{l+M}(\tilde{x}_j)$ for every $j\in\{1, \ldots, i\}$.

Assume by contradiction that $T_s^n(c)$ and $T_{\tilde{s}}^n(c)$
are on the different sides of $c$ in $[c_2, c_1]$. Since $\pi_k(m_n)=T_s^n(c)$ and $\pi_{l+M}(\tilde{m}_n)=T_{\tilde{s}}^n(c)$, by assumption
$c\in\pi_k((x_i, m_n))$ and $c\notin\pi_{l+M}((\tilde{x}_i, \tilde{m}_n))$ or the opposite. The inductive hypothesis gives $c\in \pi_k((x_j, x_{j+1}))$ if and only if  $c\in \pi_{l+M}((\tilde{x}_j, \tilde{x}_{j+1}))$ for all $j\in\{1, \ldots, i-1\}$.
Apply the shift to $[m_{n-1}, m_n]$ and $[\tilde{m}_{n-1}, \tilde{m}_n]$ and count the number of $k$-points in
$\sigma([m_{n-1}, m_n])=[m_n, m_{n+1}]$ and the number of $(l+M)$-points in $\sigma([\tilde{m}_{n-1}, \tilde{m}_n])=[\tilde{m}_n, \tilde{m}_{n+1}]$.
Every point of $k$-level strictly greater than $1$ in $[m_n, m_{n+1}]$ is a shift of some  $x_j$ and every point of $(l+M)$-level greater than $1$ in $[\tilde{m}_n, \tilde{m}_{n+1}]$
is a shift of some $\tilde{x}_j$. So it suffices
to count the $k$-points of $k$-level 1  in $[m_n, m_{n+1}]$ and the $(l+M)$-points of $(l+M)$-level $1$ in $[\tilde{m}_n, \tilde{m}_{n+1}]$.
Such points are obtained as shifts of points in $[m_{n-1}, m_{n}]$ (respectively $[\tilde{m}_{n-1}, \tilde{m}_n]$) which are projected to $c$ by $\pi_k$ (respectively $\pi_{l+M}$).
The number of such points in $[m_{n-1}, m_n]$ differs by one from the number of points in $[\tilde{m}_{n-1}, \tilde{m}_n]$, because by our assumption only one of $\pi_k([x_{i}, m_n])$ and $\pi_{l+M}([\tilde{x}_{i}, \tilde{m}_n])$ contains $c$. That is,
the number of $k$-points of $k$-level
$1$ in $[m_n, m_{n+1}]=\sigma([m_{n-1}, m_n])$ is different from the number of $(l+M)$-points of $(l+M)$-level $1$ in
$[\tilde{m}_n, \tilde{m}_{n+1}]=\sigma([\tilde{m}_{n-1}, \tilde{m}_n])$ which contradicts Proposition~$\ref{prop:bijec}$.
\end{proof}

\end{document}